\documentclass{article}
\usepackage[utf8]{inputenc}

\usepackage{amsmath}
\usepackage{amsfonts}
\usepackage{amsthm}
\usepackage{mathtools}
\usepackage{biblatex}
\usepackage{csquotes}

\usepackage{geometry}
\usepackage[english]{babel}
\newtheorem{theorem}{Theorem}[section]
\newtheorem{corollary}{Corollary}[theorem]
\newtheorem{lemma}[theorem]{Lemma}

\title{Inventory Loops (i.e. Counting Sequences) have Pre-period $2\max S_1+60$}
\author{Onno M. Cain - \texttt{onno.maw.cain@gmail.com} \\ 
Sela T. Enin - \texttt{sela.ti.enin@gmail.com}}
\date{March 2020}

\begin{document}

\maketitle

\begin{center}
\end{center}
\textbf{Abstract.} An Inventory Sequence $(S_0, S_1, S_2, ...)$ is the iteration of the map $f$ defined roughly by taking an integer to its numericized description (e.g. $f(1381)=211318$ since ``$1381$" has two $1$'s, one $3$, and one $8$). Our work analyzes the iteration under the infinite base. Any starting value of positive digits is known to be ultimately periodic [1] (e.g. $S_0=1381$ reaches the 1-cycle $f(3122331418)=3122331418$). Parametrizations of all possible cycles are also known [2,3]. We answer Bronstein and Fraenkel's open question of 26 years showing the pre-period of any such starting value is no more than $2M+60$ where $M=\max S_1$. And oddly the period of the cycle can be determined after only $O(\log\log M)$ iterations.

\section{Games and Grammar}
Mathematicians (as we all know by now) play games and often do so turning sentences into numbers. For example, what happens when we describe a number in English, and numericize the description? The number 1381 for example has
\begin{center}
    \textbf{two} $1$'s, \textbf{one} $3$, and \textbf{one} $8\quad\rightarrow\quad \mathbf21\mathbf13\mathbf18$.
\end{center}
We might call ``$211318$" the child of $1381$. So what's the grandchild? The number $211318$ has
\begin{center}
    \textbf{three} $1$'s, \textbf{one} $2$, \textbf{one} $3$, and \textbf{one} $8\quad\rightarrow\quad \mathbf31\mathbf12\mathbf13\mathbf18$.
\end{center}
Going on like this we generate the ``Family Lineage" of $1381$,
\begin{center}
    \includegraphics[scale=0.27]{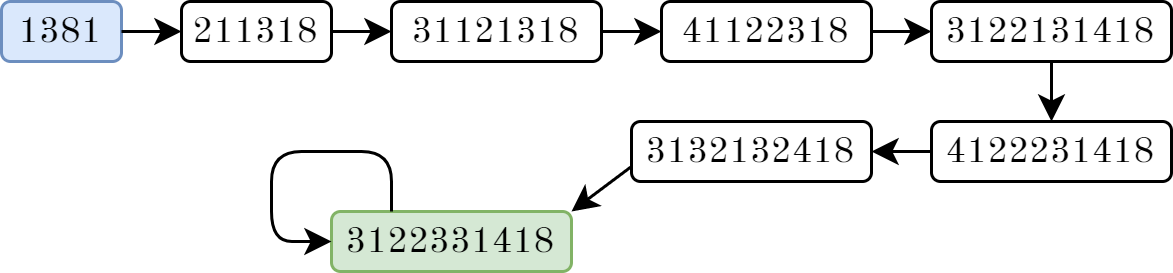}
    
    \textbf{Figure 1.1}
\end{center}
and find $\mathbf31\mathbf22\mathbf33\mathbf14\mathbf18$ which has \textbf{three} $1$'s, \textbf{two} $2$'s, \textbf{three} $3$'s, \textbf{one} $4$, and \textbf{one} $8$. We have found a number which describes itself (is the mathematical equivalent of an \textit{autogram}!). Some call these \textit{self-inventoried} numbers since they ``take inventory", so to speak, of their own digits. In our terms, we would say $3122331418$ is its own parent. So are there numbers strictly their own Grandparents? Great-grandparents? Yes and yes. 
\begin{center}
    \includegraphics[scale=0.27]{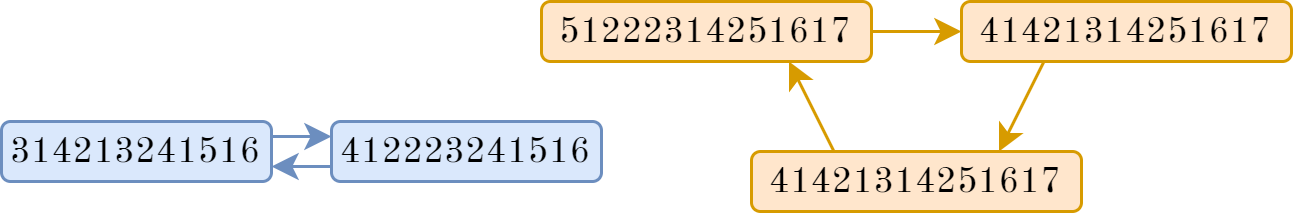}
    
    \textbf{Figure 1.2}
\end{center}
The former examples terminate the Family Lineages of $56$ and $67$ respectively. These cycles will come up repeatedly in the article. We call them \textit{Inventory Loops} since each entry \textit{takes inventory}, so to speak, of its predecessor.

Similarly to Monopoly, as one continues playing, various rule choices come up.
\begin{enumerate}
    \item When multi-digit numbers appear do we leave them clumped or split? E.g.\footnote{Thanks to math.stackexchange user Alexander Sibelius for improving this strange question's wording.} Does ``$13$" become one ``$1$" and one ``$3$" or a single ``$(13)$"?
    \item Can we pick a starting value with infinite digits?
    \item What order do we count the digits in? (notice so far digits have been counted up smallest-to-largest)
    \item Do we count \textit{all} the digits? or just consecutive runs? or even digits of all preceding ancestor?
    \item Can we skip nouns?
\end{enumerate}

Each answer choice branches a new variation and -- wonderfully! -- all the aforementioned variations have been studied. The first question is really asking whether our base is infinite (see [1,2,3]) or finite (see [5,6]). There are also perfectly reasonable infinite starting values [1]. The ``number"\footnote{Since ``$(10)$" is acting as a single ``digit" it is more technically correct to call this a \textit{sequence} rather than a \textit{number}.} $2132231455565758696(10)...$ has
\begin{center}
    two $1$'s, three $2$'s, two $3$'s, one $4$, five $5$'s, five $6$'s, five $7$'s, five $8$'s, six $9$'s, six $10$'s, ...
\end{center}
Digits can be counted smallest-to-largest (see [1,2,3,8,9]), by appearance (see [5,6,7]), or largest-to-smallest (which isn't much different from smallest-to-largest). And we have been counting all digits at once (see [1,2,3,6,8,9]) but the most famous variation -- Conway's look-and-say sequence [4,7] -- counts only consecutive runs. Counting the digits of all ancestor has even been played (see OEIS:A060857). Lastly, one can skip nouns (see [8,9]) obtaining $\mathbf{6210001000}$ which has
\begin{center}
    \textbf{six} $0$'s, \textbf{two} $1$'s, \textbf{one} $2$, \textbf{zero} $3$'s,  \textbf{zero} $4$'s, \textbf{zero} $5$'s, \textbf{one} $6$, \textbf{zero} $7$'s, \textbf{zero} $8$'s, and \textbf{zero} $9$'s.
\end{center}
These and other variations can be found on the Online Encyclopedia of Integer Sequences [10] (see A098155, A083671, A005151, A047842, A007890, A005150, A006711 for a sampling).

Researching any of these variations is confusing at first as many of them receive multiple names and some names are used of multiple variations. That is, the name$\rightarrow$variation map is ill-defined in both directions. Our variation of present concern is studied in [1,2,3]:
\begin{enumerate}
    \item infinite base (=clumping multi-digits),
    \item finite amount of digits,
    \item counting smallest-to-largest
    \item all digits of the parent only,
    \item and keeping nouns.
\end{enumerate}
Although, the order of counting doesn't have much effect as we'll see later.

Bronstein and Fraenkel showed in 1994 [1] all starting values of positive digits eventually reach an Inventory Loop. They asked a few follow-up questions. One about a different variation was answered by Sauerberg and Shu in 1997 [3]. The only question about our present variation was: how long will a given starting value take to reach an Inventory Loop? (in their terms: what are ``meaningful pre-period bounds"?). British mathematician \textit{The Math Book} by Clifford Pickover mentions Roger Hargrave also did work on our variation but Mr. Pickover unfortunately no longer has the source \footnote{Correspondence with the author.}. That is all to say, the pre-period (or ``time-to-loop") has remained unbounded for 26 years.

Here, we replicate the enumerations of Inventory Loops in [1,2,3] by alternate methods leading to a pre-period bound of $2M+60$ where $M=\max S_1$ is either the largest entry of our starting value or no greater than its length.

\section{Clarifying Names}
Our variation alone has appeared under the names ``Counting Sequences" [1,2,3], ``Reverse-likeness Sequences" [4], the ``Pea-pattern Sequence" [5], and ``Inventory Sequences" [6]. Technically, the first of the names is the only source in which the infinite base variation is rigorously defined (the base is either finite or undefined in the latter three sources). It is therefore tempting to leave our particular variation under the name ``Counting Sequences" but the author hesitates to do so for two reasons.
\begin{enumerate}
    \item Self-descriptive numbers [8,9] are sometimes called ``Self-counting" numbers and thus we would have ambiguity between the nounless [8,9] and nounful [1,2,3,4,5,6,7] variations.
    \item The adjective ``Counting" is in general hideously ambiguous in mathematics.
\end{enumerate}
Out of the aforementioned names, ``Inventory" is intuitive and has relatively little existing usage in the literature. It is therefore proposed to call both the finite-base and infinite-base variations of our iteration ``Inventory Sequences" and the corresponding cycles ``Inventory Loops". For example we would say:  ``Kowacs [5] enumerates Inventory Loops in bases $2$ through $9$" -- and -- ``We analyze Inventory Loops in the infinite base" -- etc.

\section{Multisets}
Our analysis will be done in the language of multisets -- that is (as one may guess) sets with multiple copies of elements. This guts much redundancy from the start since $f$ maps integers/sequences to the same value regardless of their digits/elements' ordering (e.g. $f(1381)=f(1138)=f(8113)=...$). To build up $f$ in the multiset ecosystem some symbols must be first defined. Let
\begin{itemize}
    \item $\mathbb{N}_+=\{1,2,3,...\}$ denote positive integers,
    \item $\mathbb{N}_+^*$ denote multisets of positive integers (e.g. $S=\{1,1,3,8\}\in\mathbb{N}_+^*$),
    \item ``$[R]$" denote the \textit{set} of elements in $R$ (e.g. $[S]=\{1,3,8\}$),
    \item ``$\text{mult}_R(x)$"\footnote{Note a multiset $R$ is also a set if $\text{mult}_R:\mathbb{N}_+^*\rightarrow\{0,1\}$.} denote ``the multiplicity of $x$ in $R$" (e.g. $\text{mult}_S(1)=2,\ 
    \text{mult}_S(4)=0$), and
    \item ``$\mu(R)$" denote ``The multiset of multiplicities in $R$". I.e.
    $$\mu(R)=\{\text{mult}_R(x):x\in[R]\}$$
    So for example $\mu(\{1,1,3,8\})=\{1,1,2\}$.
    \item For multisets $A,B\in \mathbb{N}_+^*$, the sum ``$A+B$" will mean concatenation (essentially adding multiplicities). I.e.
    $$\text{mult}_{A+B}(x)=\text{mult}_A(x)+\text{mult}_B(x).$$
    So for example $\{1,3,8\}+\{1,1,2\}=\{1,1,1,2,3,8\}$. This generalizes set union.
    \item And ``$A^B$" will mean ``the elements of $A$ also in $B$". Thus if we have $A=\{1,1,2\}$ and $B=\{1,3,8\}$ then $A^B=\{1,1\}$. This generalizes set intersection.
    \item Similarly ``$A^{\neg B}$" will mean ``the elements of $A$ \textit{not} in $B$". Thus $A^{\neg B}=\{2\}$ and $B^{\neg A}=\{3,8\}$. This -- as may be expected -- generalizes set difference.
\end{itemize}
With the new vocabulary our iteration of interest $f:\mathbb{N}_+^*\rightarrow \mathbb{N}_+^*$ becomes
$$f(S)=[S]+\mu(S).$$
For example, $f(\{1,1,3,8\})=\{1,1,1,2,3,8\},\ f^2(\{1,1,3,8\})=\{1,1,1,1,2,3,3,8\},\ $ etc. 

In addition to these symbols, we also use metaphor throughout to guide intuition. The metaphorical terms will be capitalized from here on to avoid ambiguity and are as follows\footnote{The authors wonder if all mathematicians shouldn't enumerate their metaphors as with definitions.}
\begin{itemize}
    \item $f(S)$ is thought of as the \textit{Child} of $S$ and thus $f$ as the Parent$\rightarrow$Child map ($f^2(S)$ is thought of as the \textit{Grandchild} and so on).
    \item Accordingly $S_0\overset{f}{\rightarrow}S_1\overset{f}{\rightarrow}S_2\overset{f}{\rightarrow}...$ is thought of as the \textit{Family Lineage} of $S_0$,
    \item Any $R\in\mathbb{N}_+^*$ such that $f^k(R)=S$ for some $k>0$ is thought of as an \textit{Ancestor} of $S$, and all such Ancestors are thought of as the \textit{Family Tree} or \textit{Ancestry} of $S$.
    \item $S_0$ is thought of as the \textit{First Generation}, $S_1$ as the \textit{Second Generation}, and so on.
    \item The largest element $\max S$ is thought of as the \textit{Height} of $S$. Thus if $\max f(S)>\max S$ we say $f(S)$ has \textit{Grown Taller} than its Parent.
    \item The set of multiplicities $\mu(S)$ is thought of as the \textit{Adjectives} of $S$ and the set $[S]$ as the \textit{Nouns} of $S$.
    \item Later on, some particular Adjectives which are neither the smallest or the largest of a multiset will be called the \textit{Core Adjectives}.
    \item Also later on, some multisets are obtained from others by removing or decreasing elements. This is thought of as \textit{Deterioration} and the resulting multiset as a \textit{Deteriorate}.
\end{itemize}

Lastly, for easier reading we sometimes represent multisets as integers (e.g. ``$4142x37y$" stands for ``$\{4,1,4,2,x,3,7,y\}$"). Parenthesis are placed where ambiguity requires (e.g. ``$113777(12)(77)$" stands for ``$\{1,1,3,7,7,7,12,77\}$"). This will be specified when used as ``Integer Notation".

\section{Establishing Cycles}
This section we 1) specify multisets Taller than their Parents, 2) deduce therefrom all Inventory Sequences reach a maximum Height, and 3) conclude a Loop results in all cases. Presume our Inventory Sequence $\{S_i\}_{i=0}^\infty$ (so in other words $S_0\overset{f}{\rightarrow}S_1\overset{f}{\rightarrow}S_2\overset{f}{\rightarrow}...$). We take the first part first.

\begin{lemma}
    If $\max S_{i+1}>\max S_i$ and $i\ge 1$ then $S_{i-1}=m\{1,...,n\}$ where $1\le m\le n$. In other words if a multiset past the first Generation is Taller than its Parent, we can nicely parametrize its Grandparent in two variables.
\end{lemma}
\begin{proof}
    The plan is to nail down $[S_{i-1}]$ with one bound, nail down $\mu(S_{i-1})$ with a second, and then put them together giving a form for $S_{i-1}$ itself.

    Since $i\not=0$, $S_{i-1}$ exists and $S_i=[S_{i-1}]+\mu(S_{i-1})$. Let $n=|[S_{i-1}]|=|\mu(S_{i-1})|$, the amount of distinct elements in $S_{i-1}$. It follows by pigeon-holing some element of $S_{i-1}$ is at least $n$ (just as if you had ten people in a room all different ages some person must be at least ten years old). And by the inclusion $[S_{i-1}]\subseteq S_i$ we can know 
    $$\max S_i\ge \max S_{i-1}\ge n$$
    with equality throughout when $[S_i]=[S_{i-1}]=\{1,...,n\}$.
    
    Next notice $\max S_{i+1}$ must be an Adjective of $S_i$ (otherwise $\max S_{i+1} \in S_i$). So $\max S_{i+1}=\text{mult}_{S_i}(m)$ for some $m\in S_i$. Because $[S_{i-1}]$ by definition contains only distinct values we know $\text{mult}_{S_i}(m)\le |\mu(S_{i-1})|+1=n+1$. Thus
    $$\max S_{i+1}=\text{mult}_{S_i}(m)\ge n+1$$
    with equality when $\mu(S_{i-1})=n\{m\}$ for some $m\in[S_{i-1}]$.
    
    Taking our inequalities together,
    $$n+1\ge \max S_{i+1} > \max S_i \ge n$$
    implies equality in both cases. Consequently $[S_{i-1}]=\{1,...,n\}$ and $\mu(S_{i-1})=n\{m\}$ which together force $S_{i-1}=m\{1,...,n\}$. The constraint ``$m\in S_{i-1}$" becomes ``$1\le m\le n$".
\end{proof}

\begin{corollary}
    The Height of a multiset is either equal to, or the increment of, its Parent's Height (excepting only when the Parent is the starting value).
\end{corollary}

\begin{lemma}
    If $i\ge 1$ then $|S_{i+1}|\ge |S_i|$. In other words, a Child is always larger (in order) than its Parent except possibly only when the Parent is the starting value.
\end{lemma}
\begin{proof}
    The inclusion $[S_{i-1}]\subseteq S_i$ tells us $|[S_i]|\ge |[S_{i-1}]|$. The lemma follows from $|S_{i+1}|=2|[S_i]|$ and $|S_{i}|=2|[S_{i-1}]|$
\end{proof}

\begin{lemma}
    If $\max S_{i+1}>\max S_i$ and $i\ge 2$ then either $S_{i-1}=\{1,...,n\}$ or $S_{i-1}=2\{1,...,n\}$.
\end{lemma}
\begin{proof}
    Lemma 4.1 applies so $S_{i-1}=m\{1,...,n\}$ for some $1\le m\le n$. Consequently $S_i=[S_{i-1}]+\mu(S_{i-1})=\{1,...,n\}+n\{m\}$. The bound $i\ge 2$ implies $S_{i-1}$ is not the starting value -- and therefore Lemma 4.2 gives us
    $$|S_i|=2n\ge mn=|S_{i-1}|$$
    implying $m=1,2$.
\end{proof}

\begin{lemma}
    If $\max S_{i+1} > \max S_i$ and $i\ge 3$ then $S_{i-1}$ is one of 
    $$\{1,2\},\quad \{1,2,3,4\},\quad \{1,2,3,4,5,6\},\quad \{1,1,2,2,3,3\}.$$
    In other words, if a multiset past the fourth Generation is Taller than its Parent, its Grandparent must be one of the former multisets.
\end{lemma}
\begin{proof}
    Lemma 4.3 narrows $S_{i-1}$ down to two possibilities.
    
    In the case $S_{i-1}=\{1,...,n\}$ the number $n$ must be even since $|S_{i-1}|=n=2|[S_{i-2}]|$. We will bound $n$ by the fact $$|S_{i-2}|=\sum_{x\in \mu(S_{i-2})} x.$$
    Since $\mu(S_{i-2})\subseteq S_{i-1}=\{1,...,n\}$ and since $\mu(S_{i-2})$ contains $\frac{n}{2}$ elements the former tells us $|S_{i-2}|\ge 1+2+...+\frac{n}{2}=\frac{n(n+2)}{8}$ (the latter equality is the triangular number formula $T_k=\frac{k(k+1)}{2}$ with $k=\frac{n}{2}$). Using Lemma 4.2 again,
    $$n=|S_{i-1}|\ge |S_{i-2}|\ge \frac{n(n+2)}{8}$$
    implies $\frac{n+2}{8}\le 1$ or equivalently $n\le 6$. Thus in this first case $n\in \{2,4,6\}.$
    
    Alternatively if $S_{i-1}=2\{1,...,n\}$ we create a similar bound using
    $$2n=|S_{i-1}|\ge|S_{i-2}|\ge 1 + 2+...+n=\frac{n(n+1)}{2}$$
    yielding $n\le 3$. But we can do better. The multisets $S_{i-1}=\{1,1\}$ and $S_{i-1}=\{1,1,2,2\}$ (corresponding to $n=1$ and $n=2$) have no Grandparents because their Parents,
    $$\{1\}\quad\text{and}\quad \{1,1,2\},\ \{1,2,2\},$$
    are all odd in order. But $S_{i-1}$ must have Grandparents by $i\ge 3$ implying $n\not=1,2$. This leaves $n=3$ and $S_{i-1}=\{1,1,2,2,3,3\}.$
\end{proof}

We are ready for part 2. The Family Trees of the four possible $S_{i-1}$'s of the previous lemma were worked out in full. And only three Family Trees are actually needed since $\{1,2\}$ appears in $\{1,1,2,2,3,3\}$'s Ancestry.

\begin{center}
    \includegraphics[scale=0.24]{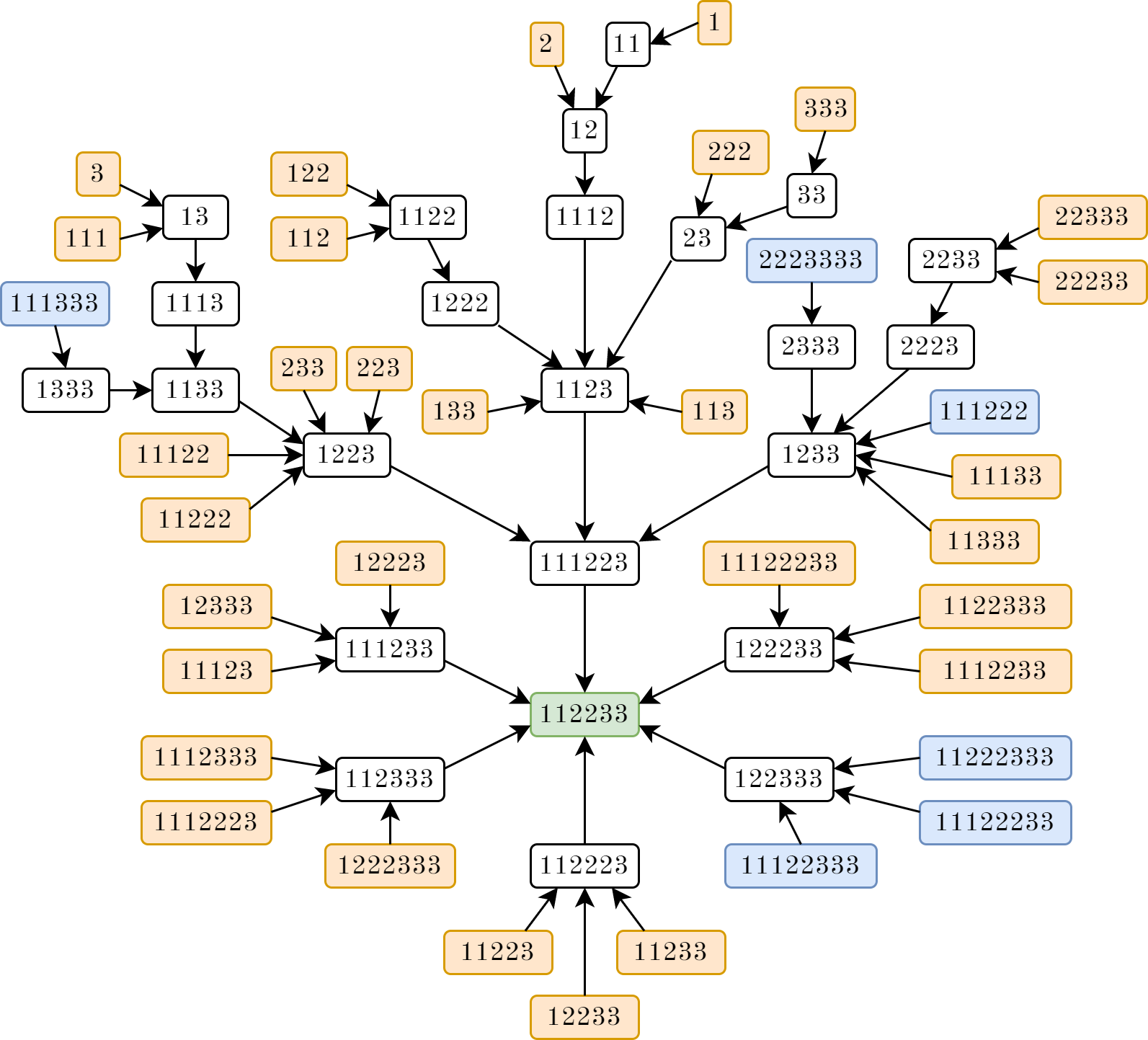}
    
    \textbf{Figure 4.1}
    
    \ 
    
    \includegraphics[scale=0.27]{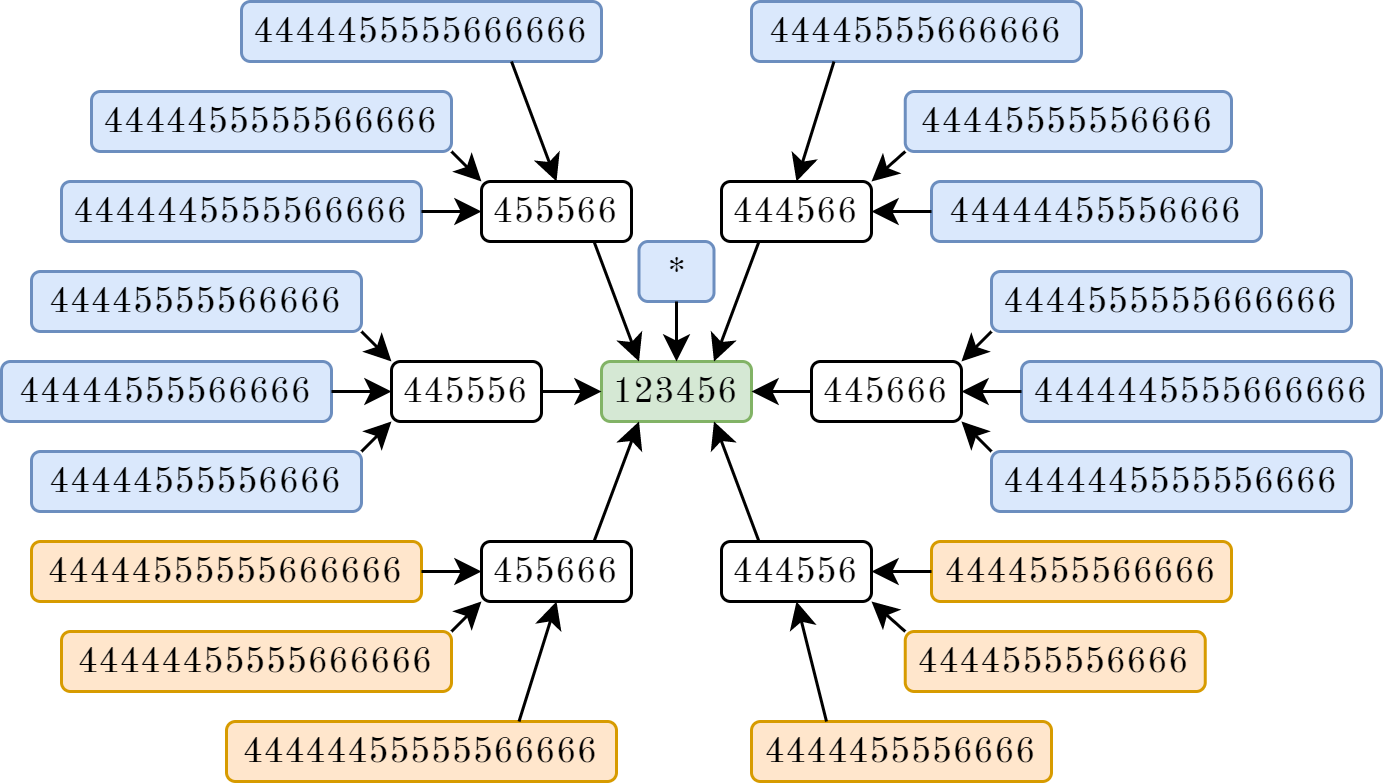}
    
    \textbf{Figure 4.2}
    
    \ 
    
    \includegraphics[scale=0.27]{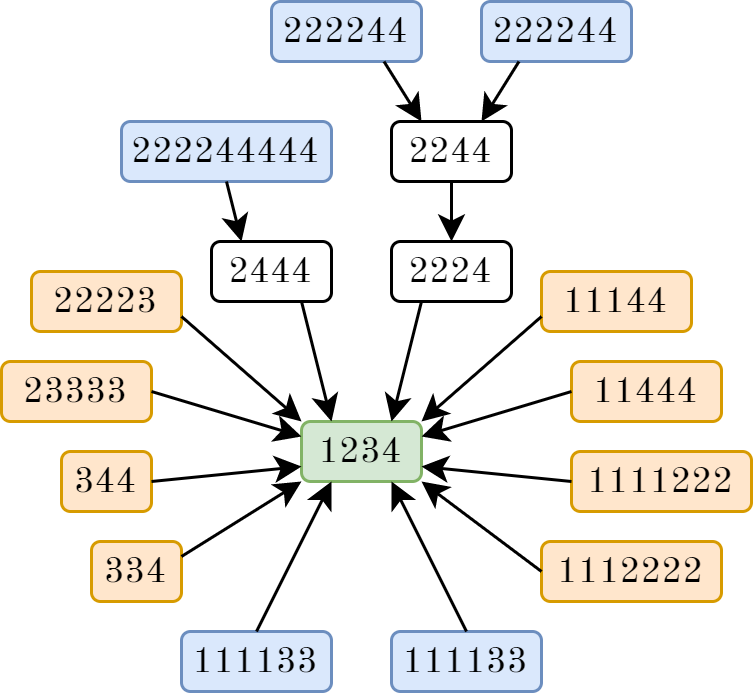}
    
    \textbf{Figure 4.3}
    
    Integer Notation is being used (e.g. ``$111333$" stands for ``$\{1,1,1,3,3,3\}$"). Orange nodes are odd in order (and therefore have no Parents). Blue nodes have so many elements Nouns cannot be selected distinctly (again meaning no Parents).
\end{center}

The figure lets us enumerate all starting values with Tallest Descendants appearing past the fourth Generation. 

\begin{theorem}
    For all $S_0\in\mathbb{N}_{+}^{*}$ the maximum Height appears by $S_8$ and, for all but finitely many $S_0$, by $S_3$ (I.e. $\max S_i = \max S_8$ for $i\ge 8$). The exceptions are given exhaustively by the entries (and the Descendants) of
    \begin{center}
        \begin{tabular}{c|c}
            max at & $S_0$ \\ \hline
            $S_4$ & $1112223^*,\ 22224444,\ 4444445555556666^{**}$ \\ \hline
            $S_5$ & $11122,\ 111222,\ 111333,\ 11222,\ 113,\ 11333,\ 133,\ 222244,\ 223,\ 224444,\ 233$ \\ \hline
            $S_6$ & $222,\ 222333$ \\ \hline
            $S_7$ & $111333,\ 112,\ 122,\ 2,\ 22233,\ 22333,\ 333$ \\ \hline
            $S_8$ & $1,\ 111,\ 3$ \\ 
        \end{tabular}
    \end{center}
\end{theorem}

The entries are written as integers for easier reading (e.g. ``$11333$" stands for $\{1,1,3,3,3\}$).

* The entry ``$1112223$" in particular stands for a Family of $15$ multisets: the Grandparents of ``$112233$" which themselves have no Parents. The Family is also equivalent to $f^{-2}(\{1,1,2,2,3\})-f^{-1}(1,1,1,2,2,3)$.

** The entry ``$4444445555556666$" in particular stands for a Family of $18$ multisets: the Grandparents of ``$123456$".

\begin{proof}
    Lemma 2.4 ensures all such exceptions must pass through one of the four given $S_{i-1}$ values. The Family Trees given above are complete since a multiset has no Parent if
    \begin{enumerate}
        \item its order is odd or
        \item there are too many elements to assign Nouns distinctly (i.e. if $|S|>2|[S]|$).
    \end{enumerate}
\end{proof}

Thus ends part 2 of this section. Part 3 is the shortest.

\begin{corollary}
    The Inventory Game ends in a Loop for all $S_0\in\mathbb{N}_{+}^*$.
\end{corollary}
\begin{proof}
    Since $\max S_i=\max S_8$ for $i\ge 8$, the proceeding multisets are bounded in length (in particular $|S_i|\le 2\max S_8$). Thus the Descendants can take on only $(\max S_8)^{2\max S_8}$.
\end{proof}

\begin{corollary}
    This gives us our first (and worst) pre-period bound: $O(M^{2M})$ where $M=\max S_1.$
\end{corollary}
\begin{proof}
    Since Height at most increments we know $\max S_8\le \max S_1 +7\in O(M)$. Past $S_8$, all multisets are one of $M^{2M}$ values.
\end{proof}

\section{Enumerating Cycles}

It turns out easier to create another game having cycles corresponding to the those of the Inventory Game and then to find all the cycles in the new game. This new game is played on the Adjectives of the Inventory Game.

Firstly, we define the new Parent-Child relationship with two functions:
\begin{enumerate}
    \item $\mu_+(S)=\{\text{mult}_S(x)+1:x\in[S]\}=\{m+1:m\in\mu(S)\}$
    \item $g_n(S)= \mu_+(S)+(n-|\mu_+(S)|)\{1\}$
\end{enumerate}

\begin{lemma}
    Every cycle under $f$ corresponds to a cycle under some $g_n$. In particular if 
    $$S_0\overset{f}{\rightarrow}S_1\overset{f}{\rightarrow}...\overset{f}{\rightarrow}S_k=S_0$$
    then
    $$\mu(S_0)\overset{g_n}{\rightarrow}\mu(S_1)\overset{g_n}{\rightarrow}...\overset{g_n}{\rightarrow}\mu(S_0)=\mu(S_k).$$
    with $n=|[S_0]|$
\end{lemma}
\begin{proof}
The rule $S_{i+1}=f(S_i)=[S_i]+\mu(S_i)$ gives us an inclusion chain
$$[S_0]\subseteq[S_1]\subseteq...\subseteq[S_k]=[S_0]$$
forcing $[S_0]=[S_1]=...=[S_k]$. So let $R=[S_0]=...=[S_k]$ and we get a new rule
$$S_{i+1}=R+\mu(S_i).$$ 
For easier further reading, we use a change of variables: $A_i=\mu(S_i)$ (with $A_i$ for $A$djectives!). The the new rule becomes $S_{i+1}=R+A_i$. Taking this together with the fact that $[A_i]\subset [S_{i+1}] = R$ we deduce
$$A_{i+1}=\mu(S_{i+1})=\mu(R+A_i)=\mu_+(A_i)+k\{1\}$$
where $k=|R|-|[A_i]|$. Letting $n=|R|$ and substituting $|[A_i]|=|\mu(A_i)|=|\mu_+(A_i)|$ we have
$k=n-|\mu_+(A_i)|$ which by definition means $A_{i+1}=g_n(A_i)$.
\end{proof}

To further investigate $g_n$ we should place it in its natural habitat. Let $T_n^*\subset \mathbb{N}_+^*$ be the set of all order $n$ multisets whose totals are twice their order. In other words
$$T_n^*=\{S\in \mathbb{N}_+^*:2|S|=2n=\sum_{x\in S}x\}.$$

\begin{lemma}
    Our function $g_n$ sends any multiset of $n$ elements into $T_n^*$. That is, if $|S|=n$ then $g_n(S)\in T_n^*$.
\end{lemma}
\begin{proof}
    Suppose $|S|=n$. Noting firstly $|\mu_+(S)|=|[S]|$, the order
    $$|g_n(S)|=|\mu_+(S)|+n-|[S]|=n$$
    is correct. Proceedingly observe
    $$\sum_{m\in\mu_+(S)}m=|\mu_+(S)|+\sum_{m\in\mu(S)}m=|[S]|+n.$$
    From this it follows the total,
    $$\sum_{x\in g_n(S)}x=n-|\mu_+(S)|+\sum_{m\in\mu_+(S)}m\quad =2n,$$
    is correct as well.
\end{proof}

\begin{corollary}
    Every cycle of multisets with order $2n$ corresponds to a cycle in $T_n^*$ under $g_n$.
\end{corollary}

Accordingly here are graphs displaying the action of $g_n$ on $T_1^*,...,T_7^*$:
\begin{center}
    \includegraphics[scale=0.27]{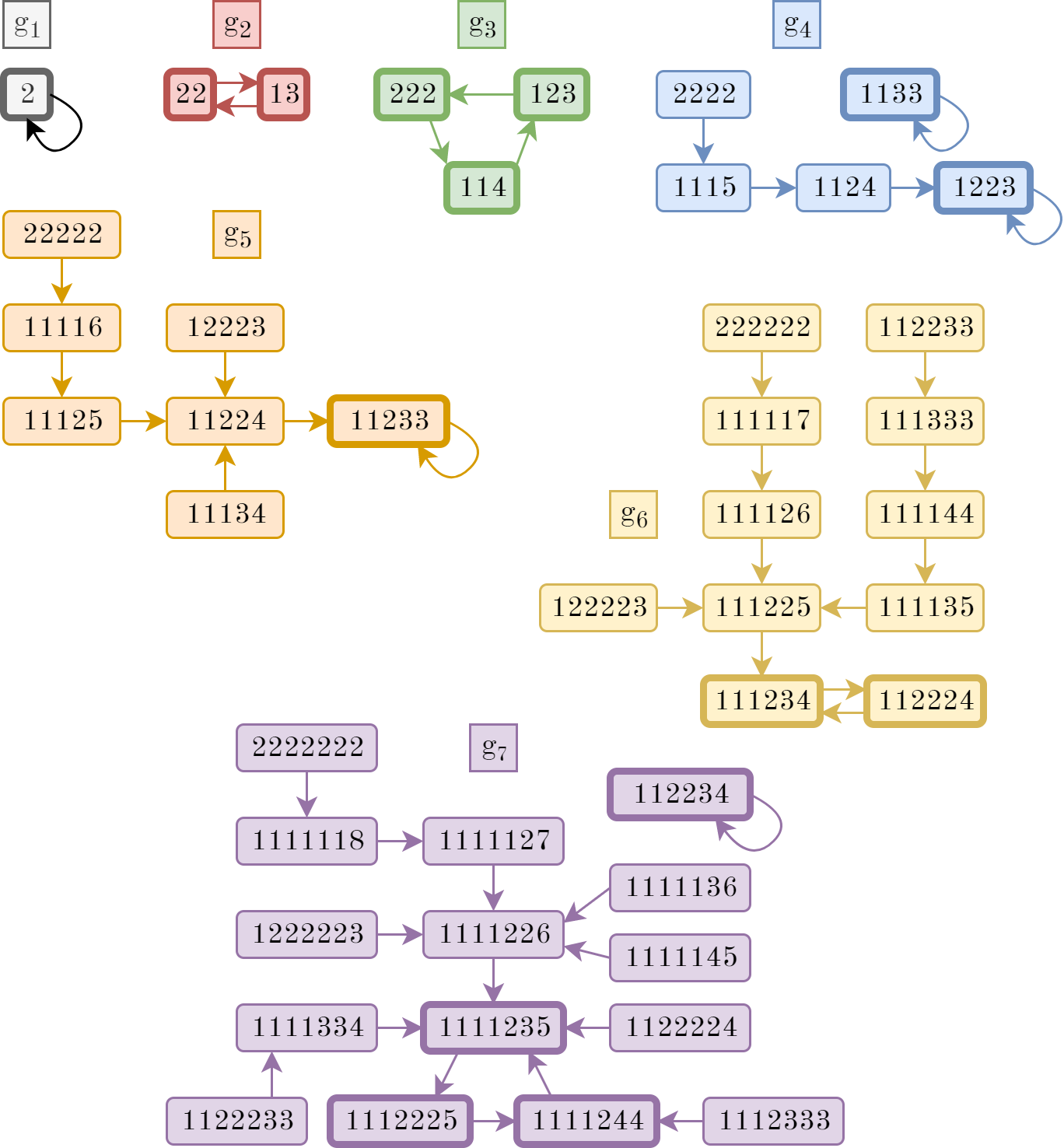}
    
    \textbf{Figure 5.1}
    
    The cycles are marked with thicker cell borders and once again we use integers to represent multisets (e.g. ``$111225$" stands for ``$\{1,1,1,2,2,5\}$").
\end{center}

\begin{theorem}
The Inventory Loops of $14$ elements or less are given exhaustively by 
\begin{center}
    \begin{tabular}{c|c}
       No. elements  & Loop(s) \\ \hline
        $2$ & $...\rightarrow 22\rightarrow ...$ \\ \hline
        $8$ & $...\rightarrow 2132231a\rightarrow ...$\\
        & $...\rightarrow 31331a1b\rightarrow ...$\\ \hline
        $10$ & $...\rightarrow 3122331a1b\rightarrow ...$\\ \hline
        $12$ & $...\rightarrow 314213241a1b\rightarrow 412223241a1b\rightarrow ...$\\ \hline
        $14$ & $...\rightarrow 413223241a1b1c\rightarrow ...$\\
        & $...\rightarrow 41421314251a1b\rightarrow 51221334151a1b\rightarrow 51222314251a1b\rightarrow ...$\\
    \end{tabular}
\end{center}
where $a,b,$ and $c$ can be any distinct whole numbers that aren't already in the Loop.
\end{theorem}
\begin{proof}
By the Corollary, any such Inventory loop must correspond to one of the nine cycles shown in Figure 5.1. The seven Inventory Loops above correspond to the cycles under $g_1, g_4, g_5, g_6$ and $g_7$ (in the manner laid out in Lemma 5.1). We claim the cycles under $g_2$ and $g_3$ do not correspond to any Inventory Loop.

The cycle of $g_2$ oscillates between $\{2,2\}$ and $\{1,3\}$. But this means any corresponding Inventory Loop must contain $1,2$ and $3$ and thus be at least order $6$. But the correspondence would enforce $g_2$'s Loop order $4$. Similarly, the cycle under $g_3$ suffers the same illness. Any corresponding loop would contain $1,2,3,$ and $4$ -- and would therefore be on multisets at least order $8$.
\end{proof}

We could use this same strategy on $g_8,g_9,$ and so on... But we have infinitely $g_n$ to go! We need an approach to knock out $n\ge 8$ and will create one in the next lemma (which should really be three or four lemmas but every attempt to split it felt unnatural and added confusion -- so it was left as a single super-lemma).

\begin{lemma}
    Any cycle in $T_n^*$ under $g_n$ with $n\ge 8$ contains one of
    $$\{1,...,1,2,n\},\quad \{1,...,3,n-1\},$$
    $$\{1,...,1,2,2,n-1\},\quad\{1,...,1,2,3,n-2\},$$
    $$\{1,...,1,2,2,2,n-2\},\quad\{1,...,1,2,2,3,n-3\}.$$
\end{lemma}
\begin{proof}
    Suppose our cycle
    $$A_0\overset{g_n}{\rightarrow}A_1\overset{g_n}{\rightarrow}...\overset{g_n}{\rightarrow}A_k=A_0.$$
    The total proof consists of ten subproofs each of which respectively proves
    \begin{enumerate}
        \item $|[A_{i+1}]|\le|[A_i]|+1$,
        \item $|[A_{i+2}]|\le|[\mu^2(A_i)]|+2$,
        \item if $|[A_{i+1}]|=|[A_i]|+1$ then $|[A_{i+2}]|\le 3$,
        \item if $|[A_{i+1}]|=|[A_i]|$ then $|[A_{i+2}]|\le 4$,
        \item $|[A_i]|\le 4$ for all $i$,
        \item $|[A_j]|\le |[A_{j+1}]|\le 4$ for some $0\ge j < k$,
        \item $|[A_i]|=|\{x\in A_{i+1}:x>1\}|$,
        \item if $n\ge 8$ then $\max A_{i+1}=\text{mult}_{A_i}(1)+1$,
        \item $|[A_i]|=\sum_{m\in A_{i+2}}(m-1)-\max A_{i+2}+1$, and
        \item $A_{j+2}$ is one of the multisets listed above.
    \end{enumerate}
    To summarize in English, we track the amount of distinct elements in each $A_i$ (written ``$|[A_i]|$"), find that it it drops to $4$ or less, and use the fact to whittle down a finite list of required cycle elements. We will take the first point first.
    
    1.) \textbf{The amount of distinct elements in $A_{i+1}$ is at most one more than that of $A_i$}. Using the definition of $A_{i+1}$ and the fact $1\not\in \mu_+(R)$ for any $R\in\mathbb{N}_+^*$ we have 
    $$|[A_{i+1}]|=|[\mu_+(A_i)+*\{1\}]|=|[\mu_+(A_i)]|+1\le |\mu_+(A_i)|+1=|[A_i]|+1.$$
    We use the ``$*$" as a multiset scalar to simply indicate the exact order is not needed but can be assumed greater than zero.
    
    2.) \textbf{The amount of distinct elements in $A_{i+2}$ is at most two more than that of $\mu^2(A_i)$}. We start similarly with 
    $$|[A_{i+2}]|=|[\mu_+(A_{i+1}])+*\{1\}|=|[\mu_+(A_{i+1})]|+1=|[\mu(A_{i+1})]|+1.$$
    This next part is a bit ugly. We must figure out what exactly $\mu(A_{i+1})=\mu(\mu_+(A_i)+l\{1\})$ looks like. But again, since $1\not \in \mu_+(R)$ for any $R\in \mathbb{N}_+^*$,
    $$\mu(\mu_+(A_i)+l\{1\})=\mu(\mu_+(A_i))+\{l\}=\mu^2(A_i)+\{l\}.$$
    Putting the two together, $|[A_{i+2}]|=|[\mu^2(A_i)+\{l\}]|+1\le |[\mu^2(A_i)]|+2.$
    
    3.) \textbf{If any $A_{i+1}$ has more distinct elements than $A_i$, then $A_{i+2}$ has at most $3$ distinct elements}. From subproof (1.), if $|[A_{i+1}]|=|[A_i]|+1$ then $|[\mu_+(A_i)]|=|\mu_+(A_i)|$. In other words, $\mu_+(A_i)$ -- and therefore also $\mu(A_i)$ -- contains all distinct values. It follows $\mu^2(A_i)=*\{1\}$ and therefore from subproof (2.): $|[A_{i+2}]|\le |[*\{1\}]|+2=3$.
    
    4.) \textbf{If any $A_{i+1}$ has equal amount of distinct elements as $A_i$, then $A_{i+2}$ has at most $4$ distinct elements}. Similarly, if $|[A_{i+1}]|=|[A_i]|$, then $|[\mu(A_i)]|=|\mu(A_i)|-1$. In other words, $\mu(A_i)$ contains exactly one duplicate entry and distinct elements otherwise implying
    $$\mu^2(A_i)=*\{1\}+\{2\}.$$
    It follows $|[A_{i+2}]|\le |[*\{1\}+\{2\}]|+2=4$.
    
    5.) \textbf{All $A_i$ have $4$ or less distinct elements}. From the two former subproofs, either $|[A_{i+1}]|<|[A_{i}]|$ or $A_{i+2}\le 4$. Thus eventually the distinct element counts drops to $4$. It then either remains at $4$ indefinitely (as with the self-inventoried number $21322314$), drops to $3$ or less (in which case our claim remains true), or increments to $5$. But if increments, then again, $|[A_{i+2}]|\le 3$ and the distinct element count can never again climb up above $4$. Thus at some point all $A_i$ have $4$ or less distinct elements. But we are in a cycle ($A_k=A_0$)! So \textit{at some point} really means \textit{always}.
    
    6.) Since what goes up must come down we may as well assume also: \textbf{$|[A_j]|\le |[A_{j+1}]|\le 4$ for some $0\le j<k$}.
    
    We now switch tracks. Instead of narrowing down the distinct element count further, we use what we've got to narrow down required cycle elements.
    
    7.) \textbf{The amount of non-$1$ elements in $A_{i+1}$ is the amount of distinct elements of $A_i$}. Note
    $$|\{x\in A_{i+1}:x>1\}|=|\mu_+(A_i)|=|[A_i]|.$$
    
    8.) \textbf{The largest element of $A_{i+1}$ is one greater than the count of $1$'s in $A_i$ (if $n\ge 8$).} From the previous two subproofs we know
    $$|\{x\in A_{i+1}:x>1\}|=|[A_i]|\le 4.$$
    In other words, at most $4$ of the $n$ elements in $A_{i+1}$ are not $1$'s. Equivalently, \textit{at least} $n-4$ of the $n$ elements \textit{are} $1$'s. Thus $1$'s are the majority if
    $$n-4\ge \frac{n}{2}\quad\Leftrightarrow\quad n\ge 8.$$
    And lastly, if $1$'s are majority then $\max A_{i+1}=\max \mu_+(A_i)=\text{mult}_{A_i}(1)+1.$
    
    9.) \textbf{The distinct element count of $A_i$ can be calculated from $A_{i+2}$}. Firstly, 
    $$\sum_{m\in A_{i+2}}(m-1)=\sum_{m\in \mu_+(A_{i+1})}(m-1)=\sum_{m\in \mu(A_{i+1})}m=|A_{i+1}|.$$
    Next we rework $|A_{i+1}|,$
    $$|A_{i+1}|=|\mu_+(A_i)|+|*\{1\}|=|[A_i]|+\text{mult}_{A_{i+1}}(1),$$
    which together with the former gives
    $$\sum_{m\in A_{i+2}}(m-1)=|[A_i]|+\text{mult}_{A_{i+1}}(1).$$
    The previous subproof is now needed. It tells us $\text{mult}_{A_{i+1}}(1)=\max A_{i+2}-1$. Thus by rearranging and substitution
    $$|[A_i]|=\sum_{m\in A_{i+2}}(m-1)-(\max A_{i+2}-1).$$
    
    10. \textbf{One of the following appears in any cycle under $g_n$ for $n\ge 8$}. The value of $j$ is that specified in subproof (6.). We use the formulas from subproofs (7.) and (9.) to calculate $|[A_j]|$ and $|[A_{j+1}]|$ from the elements of $A_{j+2}$.
    
    \begin{center}
        \begin{tabular}{c|c|c}
            $A_{j+2}$ & $|[A_{j+1}]|$ & $|[A_{j}]|$\\ \hline
            $\{1,...,1,n+1\}$ & $1$ & $0$ \\
            $\{1,...,1,2,n\}$ & $2$ & $1$ \\
            $\{1,...,1,3,n-1\}$ & $2$ & $2$ \\
            $\{1,...,1,2,2,n-1\}$ & $3$ & $2$ \\
            $\{1,...,1,2,3,n-2\}$ & $3$ & $3$ \\
            $\{1,...,1,2,2,2,n-2\}$ & $4$ & $3$ \\
            $\{1,...,1,2,2,3,n-3\}$ & $4$ & $4$ \\
        \end{tabular}
    \end{center}
    The first multiset, $\{1,...,n+1\}$, however cannot appear as it enforces $|[A_j]|=0$. In other words, it has no Grandparents. It's Parent is $n\{2\}$ which itself has a Parent only when $1+2+...+n=\frac{n(n+1)}{2}\le 2n$ -- or equivalently, when $n\le 3$ (look back to Figure 5.1 and observe only $g_1,g_2,g_3$ include $n\{2\}$ in their cycles). 
\end{proof}

Accordingly, here is a graph displaying the action of $g_n$ on the $6$ multisets (and one more) from the lemma:
\begin{center}
    \includegraphics[scale=0.3]{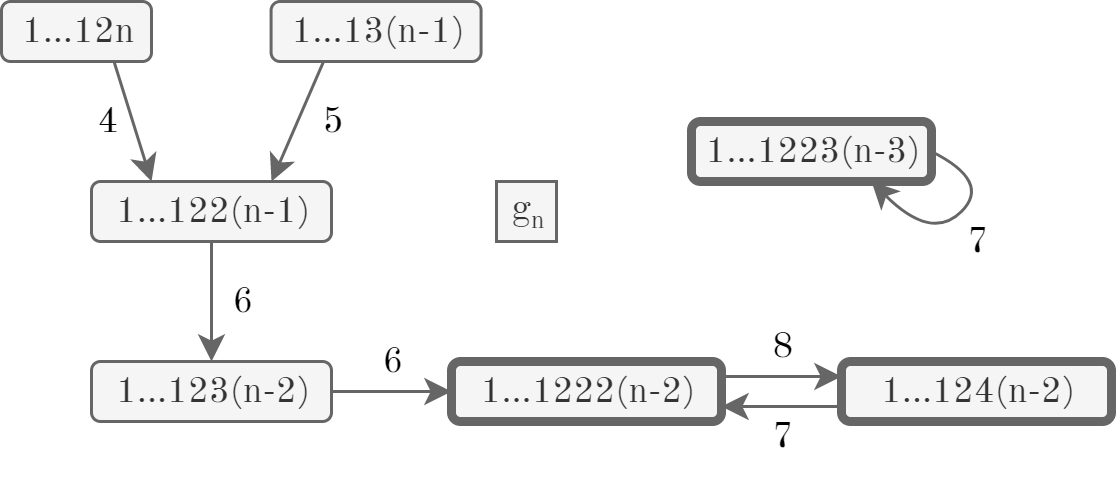}
    
    \textbf{Figure 5.2}
    
    Once again, we use integers to represent multisets (e.g. ``$1...123(n-2)$" stands for ``$\{1,...,1,2,3,n-2\}$"). The integer by each arrow is the smallest value of $n$ at and past which the map holds true.
\end{center}

\begin{theorem}
All Inventory Loops of length $n\ge 8$ fall into one of the two following parametric families:
\begin{enumerate}
    \item $...\ \rightarrow\ (n-3)132232(n-3)1a_1...1a_{n-4}\ \rightarrow\ ...$
    \item $...\ \rightarrow\ (n-3)142141(n-3)2(n-2)1a_1...1a_{n-5}\ \rightarrow\ (n-2)122242(n-3)1(n-2)1a_1...1a_{n-5}\ \rightarrow\ ...$
\end{enumerate}
where the $a_i$'s can be any whole numbers not already present in the Loop.
\end{theorem}
\begin{proof}
    Lemma 5.4 gives six multisets which must appear in any Inventory Loop and Figure 5.2 show which Loops those six multisets must fall into. The Inventory Loops above (written in Integer Notation) follow the correspondence of Lemma 5.1.
\end{proof}

\begin{corollary}
    No Inventory Loop is longer than three numbers.
\end{corollary}
\begin{proof}
    Theorems 5.3 and 5.5 together give all Inventory Loops. The longest is length $3$ appearing only at $n=7$.
\end{proof}

\begin{corollary}
    This gives us a better (but not best) pre-period bound: $O(M^2)$ where $M=\max S_1$.
\end{corollary}
\begin{proof}
    From Lemma 5.4 the value $|[A_i]|$ either decrements or falls below $4$ as $i$ increases \textit{on the assumption} all elements have appeared. Thus after $O(|[A_i]|)$ iterations, either a loop has been entered or a new element appears. Since, from the previous section, $\max S_i\in O(M)$ we may conclude at most $O(M)$ elements appear. Since also $|[A_i]|\le |A_i|=\frac{1}{2}|S_i|\in O(M)$ we may presume the pre-period quadratic in $M$. 
    
    This is not rigorous and the case $n<8$ has not been included. We leave a full proof undone as the bound is improved (with rigor) in the following section.
\end{proof}

\section{Tracking Distinct Adjectives}
This section and the next generalize the methods of the previous to Inventory Sequences in general (as opposed to just \textit{Loops}). The main result is the amount of distinct Adjectives of an Inventory Sequence (i.e. $|[S_i]|$) drops to $7$ or less in $\log\log$ time. Deducing this will require some husky machinery. To make matters worse, this is one of the few sections without pictures.

\begin{lemma}
    For any $S\in\mathbb{N}_+^*$ if $\delta$ is the difference between $S$'s order and the count of distinct elements (in other words, if $\delta=|S|-|[S]|$) then
    $$|[\mu(S)]|\le \frac{1+\sqrt{1+8\delta}}{2}.$$
\end{lemma}
\begin{proof}
    Let $c=|[\mu(S)]|$. Then 
    $$|S|=\sum_{m\in\mu(S)}m\ge 1+...+1+2+3+...+c.$$
    And since $|\mu(S)|=|[S]|$, we may proceed
    $$\delta=|S|-|[S]|=\sum_{m\in\mu(S)}(m-1)\ge 0+...+0+1+2+...+(c-1)=\frac{c(c-1)}{2}.$$
    Rearranging yields $c^2-c-2\delta\le 0$ from which the lemma follows by quadratic formula.
\end{proof}

The previous section was easy on us since we could assume no new elements were appearing (i.e. that $[\mu(S_i)]\in S_i$). Here the assumption doesn't hold. Whereas before we had a simple recurrence of Adjectives, $\mu(S_{i+1})=\mu_+(\mu(S_i))+*\{1\}$, here the corresponding recurrence -- the Ugly Recurrence -- becomes much messier:
$$\mu(S_{i+1})=\mu_+(\mu(S_i)^{S_i})+\mu(\mu(S_i)^{\neg S_i})+*\{1\}$$
Consequentially, some work is needed to patch the holes.

\begin{lemma}
    For any multisets $A,B$
    $$\text{(a)}\quad |[\mu(A+[B])]|\le 2|[\mu(A)]|+1,$$
    $$\text{(b)}\quad |\mu_+(A^B)+\mu(A^{\neg B})|=|[A]|,$$
    and if $|[B]|=1$ -- or in other words if $B$ is of the form $m\{x\}$ -- then
    $$\text{(c)}\quad |[\mu(A+B)]|\le |[\mu(A)]|+1.$$
\end{lemma}
\begin{proof}
    By our generalizations of set intersection and difference we may say $A=A^B+A^{\neg B}$ with $[A^B]\subseteq B$. This tells us $\mu(A+[B])=\mu_+(A^B)+\mu(A^{\neg B})+*\{1\}$ where the ``$*$" simply marks an unspecified scalar. Paired with the fact $|[R+S]|\le |[R]|+|[S]|$ we may deduce
    $$|[\mu(A+[B])]|\le |[\mu_+(A^B)]|+|[\mu(A^{\neg B})]|+1.$$
    Lastly, since $A^B$ and $A^{\neg B}$ are essentially a disjoint partition of $A$ we know $\mu(A^B),\mu(A^{\neg B})\subseteq \mu(A)$. This implies $|[\mu_+(A^B)]|=|[\mu(A^B)]|,|[\mu(A^{\neg B})]|\le |[\mu(A)]|$ which in turn implies
    $$|[\mu_+(A^B)]|+|[\mu(A^{\neg B})]|+1\le 2|[\mu(A)]|+1.$$
    
    For part (b) simply observe
    $$|\mu_+(A^B)+\mu(A^{\neg B})|=|\mu(A^B)+\mu(A^{\neg B})|=|\mu_+(A^B+A^{\neg B})|=|\mu(A)|=|[A]|.$$
    
    Taking part (c), if we have $x\not \in A$ then
    $$\mu(A+m\{x\})=\mu(A)+\{m\}$$
    and if instead $x\in A$ then
    $$\mu(A+m\{x\})=\mu(A)^{\neg \{m'\}}+\{m+m'\}$$
    where $m'=\text{mult}_A(x).$ In either case at most one element is appended to $\mu(A)$ with another possibly removed.
\end{proof}

Our next lemma generalizes subproofs (1.) through (4.) of Lemma 5.4.

\begin{lemma}
    Letting $c_i=|[\mu(S_i)]|$, in any Inventory Sequence 
    $$\text{(a)}\quad c_{i+1}\le c_i+1$$
    and if $c_{i+1}\le c_i+1-\delta'$ then 
    $$\text{(b)}\quad c_{i+2}\le 3+\sqrt{9+8\delta'}.$$
\end{lemma}
\begin{proof}
    To make reading easier we define $U_i=\mu_+(\mu(S_i)^{S_i})+\mu(\mu(S_i)^{\neg S_i})$ thus partially hiding the ugliness of the Ugly Recurrence.
    We have then $\mu(S_{i+1})=U_i+*\{1\}$. Since $|[R+T]|\le |[R]|+|[T]|$ for any multisets $R$ and $T$ we may deduce
    $$c_{i+1}=|[\mu(S_{i+1}]|\le |[U_i]|+1\le|U_i|+1.$$
    Applying Lemma 6.2b with $A=\mu(S_i)$ tells us further $|U_i|=|[\mu(S_i)]|=c_i$.
    
    And for (b) we start noting
    $$c_{i+2}=|[\mu(S_{i+2})]|=|[\mu([S_{i+1}]+\mu(S_{i+1}))]|=|[\mu([S_{i+1}]+U_i+*\{1\})]|.$$
    We can apply Lemma 6.2c (choosing $B=*\{1\}$) and then 6.2a (choosing $B=S_{i+1}$) obtaining
    $$c_{i+2}\le2|[\mu(U_i)]|+2.$$
    Now we saw earlier $c_{i+1}=|[U_i+*\{1\}]|$ implying $c_{i+1}\ge |[U_i]|$. Taking this together with $c_i=|U_i|$ and $\delta' \le c_i-c_{i+1}+1$ tell us
    $$\delta' \le |U_i|-|[U_i]|+1.$$
    Thus we apply Lemma 6.1 with $\delta =|U_i|-|[U_i]|$ obtaining 
    $$|[\mu(U_i)]|\le \frac{1+\sqrt{9+8\delta'}}{2}$$
    since $\delta'\le \delta +1$. Substituting into our bound on $c_{i+2}$ yields the lemma statement.
\end{proof}

Going forward, this bound on $c_i$'s is our fuel. We need one last lemma to convert it into a $\log\log$ bound of the pre-period.

\begin{lemma}
    Suppose $\{b_i\}_{i=0}^\infty$ obeys
    $$b_{i+2}\ge b_{i+1}+\frac{b_i^2-6b_i-8}{8}.$$
    If $b_1>b_0>8$ and in particular $b_1>\max (\frac{3}{4}b_0+4,8\Big(\frac{b_0}{8}\Big)^{\sqrt{2}})$ then for all $i\ge1$
    $$b_i>8\Big(\frac{b_0}{8}\Big)^{\sqrt{2}^i}.$$
\end{lemma}
\begin{proof}
    Firstly the sequence is increasing. This is seen inductively since if $b_{i+1}>b_i>8$ then $b_i^2-6b_i-8=b_i(b_i-6)-8>8$ and the recursive bound tells us $b_{i+2}> b_{i+1}+1$. Secondly we claim $b_{i+2}>\frac{1}{8}b_i^2$ when $b_{i-1}$ is defined -- or equivalently when $i\ge 1$. We may substitute the recursive bound into itself obtaining
    $$b_{i+2}\ge b_i+ \frac{b_{i-1}^2-6b_{i-1}-8}{8}+\frac{b_{i}^2-6b_{i}-8}{8}>\frac{b_{i}^2+2b_{i}-8}{8}>\frac{b_i^2}{8}.$$
    The case $i=0$ reduces to showing $8b_1-6b_0-8>0$ -- or equivalently $b_1>\frac{3}{4}b_0+4$ which is given by assumption.
    
    The lemma statement may now be shown through strong induction assuming more weakly $b_i\ge 8\Big(\frac{b_0}{8}\Big)^{\sqrt{2}^i}$ since the base case $i=0$ forces equality in the bound. The base case $i=1$ is given by assumption. The inductive case is
    $$b_{i+2}>\frac{b_i}{8}\ge \frac{1}{8}\Big(8\big(\frac{b_0}{8}\big)^{\sqrt{2}^i}\Big)^2=8\Big(\frac{b_0}{8}\Big)^{2\sqrt{2}^i}=8\Big(\frac{b_0}{8}\Big)^{\sqrt{2}^{i+2}}.$$
\end{proof}

\begin{theorem}
    Letting $c_i=|[\mu(S_i)]|$ there exists
    $$k\le 6+\log_{\sqrt{2}}\log_{5/4}\frac{\max(c_0,8)}{8}$$
    such that $c_k,c_{k+1}\le 7$.
\end{theorem}
\begin{proof}
    We should first establish some such $k$ exists. Lemma 6.3b tells us if $c_{i+1}>c_i$ (i.e. $\delta'=0$) then $c_{i+2}\le 6$ and if $c_{i+1}=c_i$ (i.e. $\delta'=1$) then $c_{i+2}\le 7$. It follows there eventually must be $c_k,c_{k+1}\le 7$ for some $k\ge 0$.
    
    More generally the same lemma gives us a bound on $c_{i+2}$ from above,
    $$c_{i+2}\le 3+\sqrt{9+8(c_i-c_{i+1}+1)},$$
    which after proper rearrangement gives us a bound on $c_i$ from below,
    $$c_i\ge c_{i+1}+\frac{c_{i+2}^2-6c_{i+2}-8}{8}.$$
    This is of course the bound from the previous lemma and by no coincidence. We may now climb back up from $c_k$ to $c_0$. But we need footing to get started.
    
    Let's presume our $k$ is the smallest such $k$. Therefore if they are defined $c_{k-1}\ge 8$ (since otherwise $c_{k-1},c_k\le 7$ and $k-1$ will do) and $c_{k-2}\ge 7$ (since otherwise $c_{k-1}\le c_{k-2}+1\le 7$). From here all prior $c_i$ can be bounded recursively. For example,
    $$c_{k-3}\ge 7+\frac{8^2-6(8)-8}{8}=8,\quad c_{k-4}\ge 8+\frac{7^2-6(7)-8}{8}=7.875,\quad ...$$
    Since each $c_i$ is an integer we may presume $c_{k-4}\ge 8$. Going on in this way we may backtrack from $c_k$ taking the smallest integer value at each turn:
    $$...,61660878524,68802238,701955,23344,2333,413,127,51,28,17,13,10,9,8,8,7,8,c_k,c_{k+1}.$$
    For instance, this sequence tells us if $k\ge 17$ then $c_0=|[\mu(S_0)]|\ge 61660878524$ -- or conversely, if the distinct Adjective count starts at less than $61660878524$ it will be down to $7$ or less before the 17th Generation. We now see just how quickly this metric drops. It's time to apply the previous lemma.
    
    Let's define $b_i=c_{k-6-i}$ so that $b_0\ge 10$ and $b_1\ge 13$. We have
    $$\frac{3}{4}b_0+4=11.5\quad\text{and}\quad 8\Big(\frac{b_0}{8}\Big)^{\sqrt{2}}=10.96...$$
    so Lemma 6.4 applies and 
    $$b_i>8\Big(\frac{5}{4}\Big)^{\sqrt{2}^i}.$$
    Picking $i=k-6$ in particular tells us
    $$c_0>8\Big(\frac{5}{4}\Big)^{\sqrt{2}^{k-6}}.$$
    Thus a first logarithm gives us $\log_{5/4}\frac{c_0}{8}>\sqrt{2}^{k-6}$ and a second gives
    $$k\le 6+\log_{\sqrt{2}}\log_{5/4}\frac{c_0}{8}.$$
    One last point must be made. In defining $b_i=c_{k-6-i}$ we have presumed $k\ge 6$ and further the $\log\log$ bound gives sensible results only if $c_0\ge 10$. Both hiccups are resolved exchanging ``$c_0$" for ``$\max(c_0,10)$" in the bound.
\end{proof}

\section{Tracking Core Adjectives}

Once the amount of distinct Adjectives has fallen below $7$, we can nearly determine the precise identity of the Adjectives themselves. But to see why we can only \textit{nearly} determine their identity we first should define two new terms. Let $m_i=\max \mu(S_i)$ and $R_i$ be the unique multiset such that 
\begin{enumerate}
    \item $1\not \in R_i$, and
    \item $\mu(S_i)=*\{1\}+R_i+\{m_i\}.$
\end{enumerate}
We call $R_i$ the \textit{Core Adjectives} of $S_i$ (where as $\mu(S_i)$ is just the \textit{Adjectives}). By the end of this section, we will have used Theorem 6.5 to narrow down the Core Adjectives to just eight possibilities past a certain point. As before, some machinery is needed before we can say so.

This next lemma generalizes subproofs (6.) through (9.) from Lemma 5.4.

\begin{lemma}
    If $|[\mu(S_k)]|,|[\mu(S_{k+1})]|\le l$ then $|R_{k+2}|\le l-1$ and 
    $$\sum_{r\in R_{k+2}} (r-1) \le l+1.$$ In other words, if the distinct Adjective counts of two consecutive Generations are $l$ or less, then there are at most $l-1$ Core Adjectives in the third Generation and their count subtracted from their sum is at most $l+1$.
\end{lemma}
\begin{proof}
    We must bring back the Ugly Recurrence from the previous section:
    $$\mu(S_{k+2})=\mu_+(\mu(S_{k+1})^{S_{k+1}})+\mu(\mu(S_{k+1})^{\neg S_{k+1}})+*\{1\}.$$
    From our definition of $R_k$ it then follows
    $$|R_{k+2}+\{m_{k+2}\}|=|\{x\in \mu(S_{k+2}:x>1\}|\le |\mu_+(\mu(S_{k+1})^{S_{k+1}})+\mu(\mu(S_{k+1})^{\neg S_{k+1}})|.$$
    That left hand side looks bad, but Lemma 6.2b gives us $|R_{k+2}+\{m_{k+2}\}|\le |[\mu(S_{k+1})]|\le l.$ Deducing $|R_{k+2}|\le l-1$ from here amounts to noticing $\{m_{k+2}\}$ is a singleton.
    
    The second part is trickier. It begins with the equation
    $$(m_{k+2}-1)+\sum_{r\in R_{k+2}}(r-1)=\sum_{x\in \mu(S_{k+2})}(x-1)$$
    And here the author is not sure how to proceed with clarity. It seems the Ugly Recurrence must be used once more. Note that if we replaced one ``$\mu$" in particular with a ``$\mu_+$" the Ugly Recurrence becomes much less intolerable:
    $$\mu_+(\mu(S_{k+1})^{S_{k+1}})+\mathbf{\mu_+}(\mu(S_{k+1})^{\neg S_{k+1}})=\mu_+(\mu(S_{k+1})^{S_{k+1}}+\mu(S_{k+1})^{\neg S_{k+1}})=\mu_+(\mu(S_{k+1})).$$
    In total, the change increments a few elements of $\mu(S_{k+2})$. Thus we may say
    $$\sum_{x\in \mu(S_{k+2})}(x-1)\le \sum_{x\in \mu_+(\mu(S_{k+1}))}(x-1)=\sum_{x\in \mu(\mu(S_{k+1}))}x=|\mu(S_{k+1})|.$$
    Another application of the Ugly Recurrence and Lemma 6.2b give us
    $$|\mu(S_{k+1})|=|\mu_+(\mu(S_{k})^{S_{k}})+\mu(\mu(S_{k})^{\neg S_{k}})+*\{1\}|=|[\mu(S_k)]|+\text{mult}_{\mu(S_{k+1})}(1)\le l+\text{mult}_{\mu(S_{k+1})}(1).$$
    Linking these four messes together neatly gives us
    $$(m_{k+2}-1)+\sum_{r\in R_{k+2}}(r-1)\le l+\text{mult}_{\mu(S_{k+1})}(1).$$
    Finishing off the lemma now amounts to showing $\text{mult}_{\mu(S_{k+1})}(1)\le m_{k+2}$. Note $S_{k+2}=[S_{k+1}]+\mu(S_{k+1})$ tells us $\text{mult}_{\mu(S_{k+1})}(1)\le \text{mult}_{S_{k+2}}(1)$. And since we defined $m_{k+2}=\max \mu(S_{k+2})$ we may conclude 
    $$\text{mult}_{\mu(S_{k+1})}(1)\le \text{mult}_{S_{k+2}}(1)\le m_{k+2}.$$
\end{proof}

\begin{corollary}
    If $|[\mu(S_k)]|,|[\mu(S_{k+1})]|\le 7$ then $R_{k+2}$ -- expressed in Integer Notation -- is one of
    $$2, 3, 4, 5, 6, 7, 8, 9,$$
    $$22,23,24,33,25,34,26,35,44,27,36,45,28,37,46,55,$$
    $$222,223,224,233,225,234,333,226,235,244,334,227,236,245,335,344,$$
    $$2222,2223,2224,2233,2225,2234,2333,2226,2235,2244,2334,3333,$$
    $$22222,22223,22224,22233,22225,22234,22333,$$
    $$222222,222223,222224,222233,$$
    or is the empty set $\emptyset$.
\end{corollary}
\begin{proof}
    Simply exhaust the possibilities specified by the previous lemma with $l=6$.
\end{proof}

To proceed we must study how these $64$ possibilities of the Core Adjectives map to each other under $f$. Unfortunately, $f$ does not determine their mapping \textit{uniquely}. In other words, knowing $R_i$ is not enough information to determine $R_{i+1}$. The relationship is determined in Loops but in Inventory Sequences in general knowing $R_i$ only narrows the identity of $R_{i+1}$ down to some candidate values. This indeterminacy is caused again by the same non-inclusion which made the Ugly Recurrence ugly. 

So what assumptions, then, did we make when working with Loops? Two actually:
\begin{enumerate}
    \item All new elements have appeared.
    \item $n=|\mu(S_i)|$ is sufficiently large.
\end{enumerate}
And indeed when we assume these, $f$ forces a unique map on Core Adjectives since
\begin{enumerate}
    \item if all new elements have appeared then $[R_i+\{m_i\}]\subset S_i$ and
    \item if $n$ is large enough then most Adjectives must be $1$'s (presuming the Core Adjectives are known and fixed) and thus $m_{i+1}=\text{mult}_{S_{i+1}}(1)$. 
\end{enumerate}
With such assumptions fullfilled the equation
$$\mu(S_{i+1})=\mu([S_i]+*\{1\}+R_i+\{m_i\})=*\{1\}+\{m_{i+1}\}+\mu_+(R_i)+\{2\}$$
tells us $R_{i+1}=\{2\}+\mu_+(R_i)$. Simple enough. So in accordance with our naming convention from Section 5 (where $g_n$ was the map on Adjectives) and the former two naive assumptions we define a map on Core Adjectives:
$$g_\text{naive}(R)=\{2\}+\mu_+(R).$$
Here then is $g_\text{naive}$'s action on the $64$ possibilities from the previous corollary:
\begin{center}
    \includegraphics[scale=0.25]{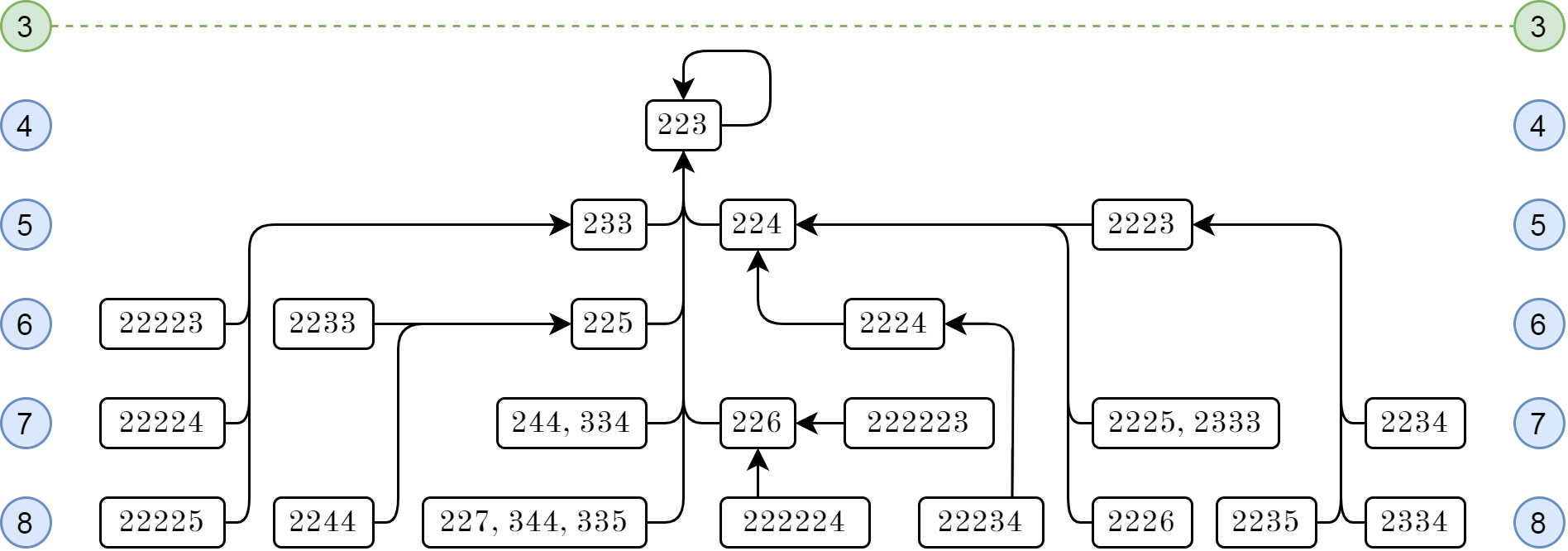}
    
    \textbf{Figure 7.1}
    
    \ 
    
    \includegraphics[scale=0.27]{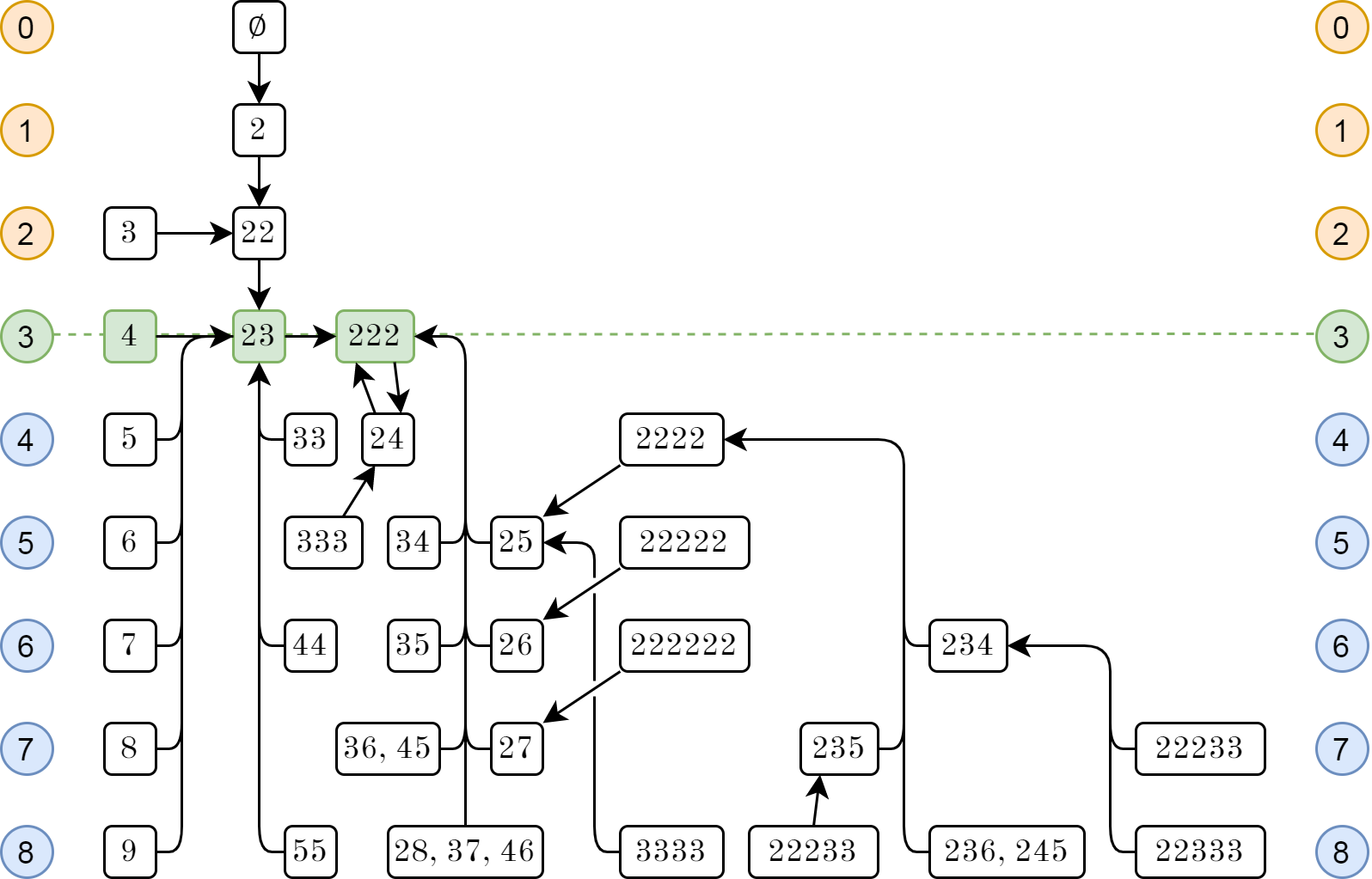}
    
    \textbf{Figure 7.2}
    
    The Core Adjectives appear in Integer Notation and are arranged vertically by $\sum_{r\in R_i}(r-1)$ as indicated on the left/right.
\end{center}

One can see a single self-cyclic node in Figure 7.1 and a loop of two nodes in Figure 7.2. As one might expect, these correspond to the parametric $1$-cycle and $2$-cycle given in theorem 5.5. And if $g_\text{naive}$ covered the general case we would be done. Unfortunately, it is $f$ which determines $R_i\rightarrow R_{i+1}$.

At this point we can specify how exactly $g_\text{naive}$ narrows down the identity of $R_{i+1}$. We must first define \textit{Deterioration}. A multiset $R$ is called a \textit{Deteriorate} of another $S$ if $R$ can be obtained from $S$ using one or both of the following operations:
\begin{enumerate}
    \item Replace any subset of the elements by their decrements discarding $1$'s (e.g. $\{2,\mathbf2,\mathbf4,5,\mathbf8\}\rightarrow\{2,\mathbf3,5,\mathbf7\}$).
    \item Replace the largest element by any lesser value or discard it entirely (e.g. $\{2,2,4,\mathbf8\}\rightarrow\{2,2,4,\mathbf4,5\}$ or $\{2,3,\mathbf5\}\rightarrow\{2,3\}$).
\end{enumerate}

\begin{lemma}
    The multiset $R_{i+1}$ is either equal to or a Deteriorate of $g_\text{naive}(R_i)$ and is equal only if $[R_i+\{m_i\}]\subseteq S_i$ and $m_{i+1}=\text{mult}_{S_{i+1}}(1)$.
\end{lemma}
\begin{proof}
    We claim the two operations of Deterioration correspond respectively to the two assumptions made in defining $g_\text{naive}$. To see this we should start by writing out the recurrence defining $R_{i}\rightarrow R_{i+1}$ in the full messy general case:
    $$\mu(S_{i+1})=\mu([S_i]+*\{1\}+R_i+\{m_i\})=*\{1\}+\{\text{mult}_{S_{i+1}}(1)\}+\mu_+(R_i^{S_i})+\mu(R_i^{\neg S_i})+
    \begin{cases} \{2\}& \text{if } m_i\in S_i \\
    \emptyset  & \text{otherwise} \end{cases}.$$
    
    Thus consider some particular $x\in R_i$. If $x\in S_i$ then $x$ might contribute a ``$\text{mult}_{R_i}(x)+1$" to $R_{i+1}$ via $\mu_+(R_i^{S_i})$ (and fails to do so only if $m_{i+1}=\text{mult}_{R_i}(x)+1$ -- i.e. only if the contribution is the largest Adjective). But if $x\not\in S_i$ then $x$ might contribute only a ``$\text{mult}_{R_i}(x)$" (and fails to do so as in the former case). And this nearly covers the first operation of Deterioration. Note lastly in the latter case if $\text{mult}_{R_i}(x)=1$ then the contribution is absorbed into $*\{1\}$. 
    
    The second operation of Deterioration covers the possibility $\text{mult}_{S_{i+1}}(1)$ might not be the largest Adjective. If not, some element of $\mu_+(R_i^{S_i})$ or of $\mu(R_i^{\neg S_i})$ will take the place of ``$m_{i+1}$" and $\text{mult}_{S_{i+1}}(1)$ will fall into $R_{i+1}$ -- or will be discarded entirely if $=1$. Thus the second operation covers all such possible exchanges (in fact, it covers more -- but its tight enough for the use we will make of it).
    
    From this the ``if" of the lemma's ``only if" follows immediately. Conversely, if $R_{i+1}=g_\text{naive}(R_i)$ then we have
    $$\{m_{i+1}\}+\mu_+(R_i)+\{2\}=\{\text{mult}_{S_{i+1}}(1)\}+\mu_+(R_i^{S_i})+\mu(R_i^{\neg S_i})+
    \begin{cases} \{2\}& \text{if } m_i\in S_i \\
    \emptyset  & \text{otherwise} \end{cases}.$$
    with all elements $\ge 2$. Equality of order dictates $m_i\in S_i$ simplifying the equation to
    $$\{m_{i+1}\}+\mu_+(R_i)=\{\text{mult}_{S_{i+1}}(1)\}+\mu_+(R_i^{S_i})+\mu(R_i^{\neg S_i}).$$
    And since $m_{i+1}$ is at least as big as anything in $\mu_+(R_i)$ we may presume either $m_{i+1}=\text{mult}_{S_{i+1}}(1)$ or $m_{i+1}\in \mu_+(R_i^{S_i})$. But in the latter case either we have also the former or the elements of the right-hand-side sum to strictly less than the sum of those on the left-hand-side. We may then presume $m_{i+1}=\text{mult}_{S_{i+1}}(1)$ and are left with 
    $$\mu_+(R_i)=\mu_+(R_i^{S_i})+\mu(R_i^{\neg S_i}).$$
    Considering the sum of elements again tells us $\mu_+(R_i)=\mu_+(R_i^{S_i})$ implying $R_i=R_i^{S_i}$ implying finally $[R_i]\subseteq S_i$.
\end{proof}

These lemmas taken together are enough to wrangle $R_i$'s list of candidate values from the infinite into the finite.

\begin{theorem}
    Letting $c_0=|[\mu(S_0)]|$, then for
    $$k= 13+\log_{\sqrt{2}}\log_{5/4}\frac{\max(c_0,10)}{8}$$
    then either
    \begin{enumerate}
        \item $S_k$ is a fixed point under $f$ (i.e. and therefore an Inventory Loop), or
        \item The Core Adjectives $R_k$ in Integer Notation are one of $2,3,4,22,23,24,222,$ or are the empty set $\emptyset$.
    \end{enumerate}
\end{theorem}
\begin{proof}
    Theorem 6.5 and Corollary 7.1.1 taken together imply there exists
    $$i\le 8+\log_{\sqrt{2}}\log_{5/4}\frac{\max(c_0,10)}{8}$$
    such that $R_i$ is one of the 64 multisets there listed. But more strongly, we make the claim not simply of $R_i$ but for all $R_j$ thereafter (i.e. that $R_j$ is one of the 64 multisets for $j\ge i$). Those $64$ multisets are clearly closed under $g_\text{naive}$ (as Figures 7.1 and 7.2 show). But they are closed also under Deterioration since $\sum_{r\in R_i}(r-1)$ is strictly decreased. Lemma 6.2 tells us then if $R_i$ is one of our 64 values, then so is $R_{i+1}$. In other words for 
    $$j\ge 8+\log_{\sqrt{2}}\log_{5/4}\frac{\max(c_0,10)}{8}$$
    the multiset $R_j$ is one of the 64 values. 
    
    Further, after adding Deteriorate edges to Figures 7.1 and 7.2 the resulting graph has exactly two self-cyclic strongly-connected-components (SCCs). This can of course be verified by computer but can also be seen more easily from the Figures. Since Deterioration strictly decreases $\sum_{r\in R_i}(r-1)$ we are assured all but 3 Deteriorate edges originating beneath the dashed green line point upwards (with exceptions $2222\rightarrow \{24,5\},\ 22222\rightarrow \{25,6\},\ $ and $222222\rightarrow\{27, 8\}$ which each point horizontally). Hence the graphs were arranged vertically as they were. 
    
    After $4$ additional iterations the sequence $\{R_i\}_{i=0}^\infty$ must therefore enter one of the two SCCs:
    \begin{enumerate}
        \item $\{\emptyset,2,3,4,22,23,24,222\}$ or
        \item $\{223\}$.
    \end{enumerate}
    If the first is entered, the theorem holds true. If instead the 2nd SCC is entered, then after another additional iteration (so after $5$ in total) we must have $R_i=R_{i+1}=R_{i+2}=\{2,2,3\}$ for some $i> 2$. It turns out this is enough information to know $S_{i+2}$ a fixed point of $f$ (and on this we spend the proof's remainder).
    
    To make things easier we presume $1\in S_i$ (and address the case $1\not\in S_i$ last). Lemma 7.2 tells us $[R_i+\{m_i\}]\subseteq S_i$. But since in addition to $m_i$ and the elements of $R_i$ the multiset $\mu(S_i)$ contains only $1$'s, the former really means $[\mu(S_i)]\subseteq S_i$ implying $[S_{i+1}]=[S_i]$ -- or in English, that no new elements have appeared. The same can be said also of $\mu(S_{i+1})$ so in total 
    $$[S_{i+2}]=[S_{i+1}]=[S_i].$$
    The definition/equation $\mu(S_j)=*\{1\}+R_j+\{m_j\}$ also tells us
    $$|[S_j]|=|\mu(S_j)|=\text{mult}_{\mu(S_j)}(1)+4$$
    for $j=i,i+1,i+2.$ Putting the former with the latter we then obtain
    $$\text{mult}_{\mu(S_i)}(1)=\text{mult}_{\mu(S_{i+1})}(1)=\text{mult}_{\mu(S_{i+2})}(1).$$
    Here again the assumption $1\in S_i$ is relevant. By it, the definition $S_{j+1}=[S_j]+\mu(S_j)$ tells us
    $$\text{mult}_{S_{j+1}}(1)=\text{mult}_{\mu(S_j)}(1)+1$$
    for $j=i,i+1$. In particular, we therefore have $\text{mult}_{S_{i+1}}(1)=\text{mult}_{S_{i+2}}(1)$ which by Lemma 7.2 again tells us $m_{i+1}=m_{i+2}$. Now if we simply put together the deductions A) $R_{i+1}=R_{i+2}$, B) $\text{mult}_{\mu(S_{i+1})}(1)=\text{mult}_{\mu(S_{i+2})}(1)$, and C) $m_{i+1}=m_{i+2}$ it follows that $\mu(S_{i+1})=\mu(S_{i+2})$. And paired with $[S_{i+1}]=[S_{i+2}]$ this of course tells us $S_{i+2}=S_{i+3}$.
    
    Lastly if $1\not\in S_i$ then $\mu(S_{i-1})$ must contain only $2$'s since if not, we would have
    $$|S_{i-1}|=\sum_{m\in \mu(S_{i-1})}m>2|\mu(S_{i-1})|=|S_i|.$$
    But Lemma 4.2 tells us a Parent can only be larger in order than its Child if the first Generation -- a contradiction since $i>2$. But then if $\mu(S_{i-1})=*\{2\}$ (i.e. if the Adjectives are all $2$'s) then $$\mu(S_i)=\mu([S_{i-1}]+*\{2\})=*\{1\}+\{\text{mult}_{S_i}(2)\}$$
    which would imply $R_i=\emptyset$ -- a contradiction as well. Thus we may (as done above) presume $1\in S_i$.
\end{proof}

\begin{corollary}
    The period of any Inventory Sequence can be determined after
    $$k= 13+\log_{\sqrt{2}}\log_{5/4}\frac{\max(c_0,10)}{8}$$
    iterations from the initial value $S_0\in\mathbb{N}_+^*$.
\end{corollary}
\begin{proof}
    This is odd since (as we will see later) the Loop itself may not begin for many iterations after $S_k$. The former Theorem tells us after $k$ iterations either $S_k$ is a fixed point (an Inventory Loop with period $1$) or the Core Adjective multiset $R_k$ is in Integer Notation one of $2,3,4,22,23,24,222$ or is the empty multiset $\emptyset$. But those eight possibilities are closed under $g_\text{naive}$ and Deterioration and thus the Core Adjectives will be one of those eight possible multisets.
    
    Out of the two parametric Loops of Theorem 5.5 only the $2$-cycle uses the eight Core Adjectives (the $1$-cycle has $R_k=\{2,2,3\}$). 
    
    And the Loops listed in Theorem 5.3 all start within the first $12$ iterations. This was checked with computers by generating family trees. The longest pre-period was $S_0=6\{6\}+7\{7\}$ which started its Loop at $S_{12}=\{1,1,1,1,1,2,2,3,3,4,5,5,6,7\}$ with period $3$.
\end{proof}

\section{The Main Stuff}

\begin{center}
    \enquote{\textit{I am to give my readers not the best absolutely but the best I have.}}
    
    - C. S. Lewis, \textit{The Problem of Pain}
\end{center}

The author found it difficult to write this section clearly even though, containing the main result, it is the most important. Any reader therefore intending to build higher should inspect this foundation in case the author fell short of proper rigor. They ask your patience.

We should begin by summarizing our position. Theorem 7.3 told us some Inventory Sequences reach a Loop in $O(\log\log)$ time. And further, Core Adjectives of the Sequences which don't are one of 
$$\emptyset, 2, 3, 4, 22, 23, 24, 222$$
past a certain point (also reached in $O(\log\log)$ time). Our remaining task then is to analyze how exactly the former eight multisets of Core Adjectives jump from one to each other. In other words, what values can $(R_i, R_{i+1})$ take?

The rules of Deterioration laid out in the previous section are enough to specify which possible values $R_{i+1}$ might take (that is what Lemma 7.2 was saying). But if our Sequence is order $16$ or larger -- equivalently, if $|\mu(S_i)|\ge 8$ -- and we know also which elements (if any) of $\mu(S_i)$ have appeared for the first time in the Sequence, then $R_i$ \textit{is} sufficient information to determine $R_{i+1}$. In other words, if there are at least eight Adjectives then the Core Adjectives and new appearances of one Generation are enough to work out the Core Adjectives of the next Generation (we will see it is also enough to work out $m_{i+1}$).

What follows is said (and understood) more easily if we define an Inventory Sequence's \textit{Maturity}. We say such a Sequence has \textit{Matured} at iteration $i$ if $|S_i|\ge 16$ and the Core Adjectives are one of the former eight possibilities. If either of these assumptions is broken we say the Sequence is \textit{Immature}. \footnote{So for example, the parametric $1$-cycle from Theorem 5.5 might well have over $16$ elements per term. However the Core Adjectives ($223$) are not one of the former eight possibilities. Thus we say such a Sequences has Looped \textit{Immaturely}.} Theorem 7.3 essentially says any Inventory Sequence either Loops or reaches Maturity in $O(\log\log |[\mu(S_0)]|)$ time.

Let's assume then our Sequence has Matured and redraw Figure 7.2 with just the Core Adjectives we care about. This time Deteriorate edges corresponding to new element appearances are shown (and are labeled by the elements making new appearance).
\begin{center}
    \includegraphics[scale=0.4]{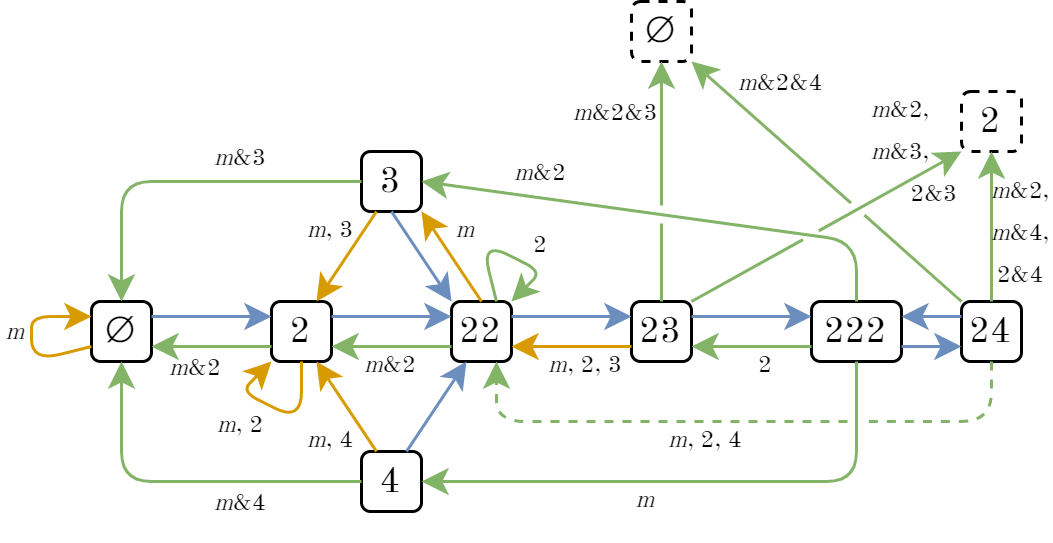}
    
    \textbf{Figure 8.1}
    
    The letter ``$m$" is shown when value of $m_i=\max \mu(S_i)$ is making a new appearance. The ``\&" denotes multiple new elements appearing together (e.g. ``$m$\&$2$\&$3$" means $m_i,2,$ and $3$ have each and all together appeared for the first time in the entire Inventory Sequence). Blue arrows are for now new appearances (and therefore match up with the black arrows of Figure 7.2). Orange arrows are for new appearance edges which can be taken any amount of times. Green arrows are for new appearance edges which can be taken only once (or, for the dashed arrow, at most twice). The dashed nodes containing $\emptyset$ and $2$ are only placeholders to reduce clutter.
\end{center}

We won't compute the correctness of all $26$ arrows by hand. Any individual arrow can be checked if needed. However the claim some edges can be taken at most once (or twice) does require longer justification.

\begin{lemma}
    If an Inventory Sequence is past Maturity then particular consecutive Core Adjective pairs $(R_i, R_{i+1})$ can further appear at most a fixed number of times -- given explicitly by the following table:
    \begin{center}
        \begin{tabular}{c|c}
            Edge(s) & Max Occurrences\\ \hline
            $222\rightarrow 3, 4, 23$ & $1$\\
            $24\rightarrow \emptyset$ & $1$\\
            $24\rightarrow 22$ & $4$\\
            $23, 24\rightarrow 2$ & $2$\\
            $23\rightarrow \emptyset$ & $1$\\
            $22\rightarrow 2, 22$ & $1$\\
            $2,3,4\rightarrow \emptyset$ & $1$\\
        \end{tabular}
    \end{center}
\end{lemma}
\begin{proof}
    Most edges occur at most once or twice because they require the new appearance of $2,3,$ and/or $4$ which by definition can happen at most once. Nothing can appear for the first time \textit{twice}. The two edges which could plausibly appear repeatedly are $222\rightarrow 4$ and $24\rightarrow 22$ requiring the new appearance of $m_i$ only.
    
    To take care of these exceptional edges, we will use Backtracking. The process is tedious to define so we begin by just showing the Backtracking Tree for $222\rightarrow 4$ and define afterwards.
    
    \begin{center}
        \includegraphics[scale=0.3]{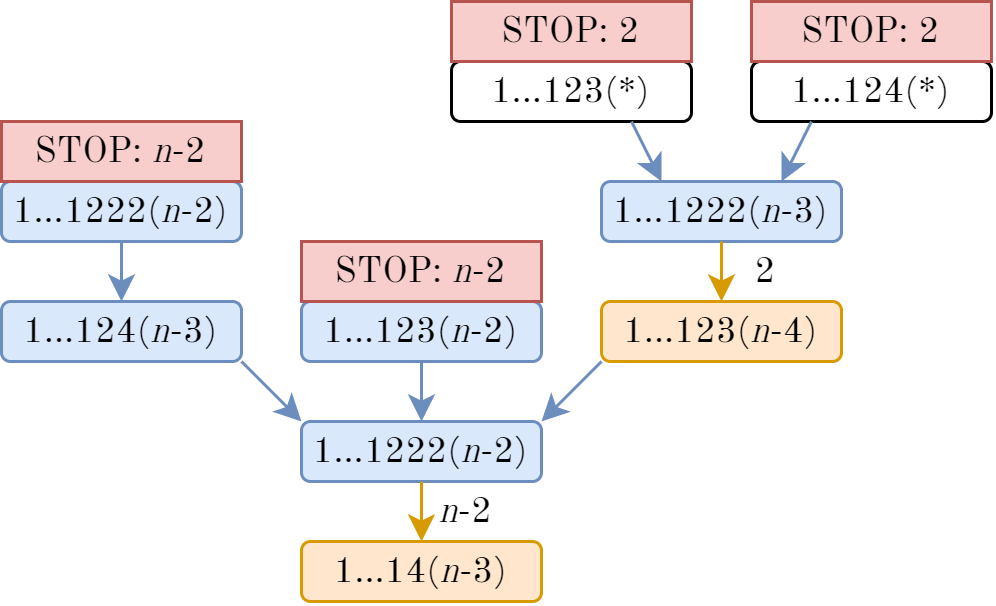}
        
        \textbf{Figure 8.2}
    \end{center}
    
   The idea is to start with $\mu(S_i),\mu(S_{i+1})$ and work backwards to possible values of $\mu(S_{i-1}), \mu(S_{i-2}), \mu(S_{i-3}),$ ... and so on. The trick is that by tracking the appearance of new elements carefully, we can terminate the Backtracking after a finite number of steps. This is done by forcing the Sequence into Immaturity (at which point we don't push farther since we have a good bound on the development of Maturity\footnote{As some of us wish we had generally.}). The tedious bit is determining what value each $m_j$ must take in terms of $n=|\mu(S_i)|$. Notice Maturity implies $n\ge 8$. Our work from Section 5 will help us. Let's take an example.
    
    If $R_i=\{2,2,2\}$ we claim either $m_i=n-2$ or the Sequence has just Matured at $S_i$. If the Sequence was also Mature at $S_{i-1}$ then there were \textit{no new appearances} in $\mu(S_{i-1})$ -- or equivalently, $[S_i]=[S_{i-1}]$. This is because the multiset $\{2,2,2\}$ receives only blue arrows in Figure 8.1. We therefore know $|\mu(S_{i-1})|=|\mu(S_i)|=n$ and $|S_i|=2|\mu(S_{i-1})|=2n$. Putting these together we can derive
    $$m_{i}+2=1+1+1+(m_i-1)=\sum_{x\in \mu(S_i)}(x-1)=|S_i|-|[\mu(S_i)]|=2n-n=n$$
    implying $m_i=n-2$. 
    
    Similarly if $k$ new elements appear we can amend the derivation starting instead with $|\mu(S_{i-1}|=|\mu(S_i)|-k=n-k$ and $|S_i|=2n-2k$. The corresponding result is $m_i=n-2k+1-\sum_{x\in R_i}(x-1).$ Accordingly let's make a table of $m_i$ values (assuming Maturity in the prior Generation):
    \begin{center}
        \begin{tabular}{c|c|c|c|c}
            $R_i$ & $k=0$ & $k=1$ & $k=2$ & $k=3$ \\ \hline
            $\emptyset$ & -- & $n-1$ & $n-3$ & $n-5$ \\
            $2$ & $n$ & $n-2$ & $n-4$ & -- \\
            $3$ & -- & $n-3$ & $n-5$ & -- \\
            $4$ & -- & $n-4$ & $n-6$ & -- \\
            $22$ & $n-1$ & $n-3$ & -- & -- \\
            $23$ & $n-2$ & $n-4$ & -- & -- \\
            $24$ & $n-3$ & -- & -- & -- \\
            $222$ & $n-2$ & -- & -- & -- \\
        \end{tabular}
    \end{center}
    The em dashes ``--" mark impossible iterations past Maturity. For example, $R_i=\{24\}$ and $k=1,2,3$ are marked because past Maturity $R_i=\{2,2,2\}$ appears only when no new elements have appeared ($222$ receives only blue arrows in Figure 8.1). Similarly, only $R_i=\emptyset$ is unmarked in the $k=3$ column because only $\emptyset$ receives arrows requiring $3$ new elements.
    
    There is another picky point to be made about Backtracking. Once we have fixed $n=|\mu(S_i)|$ for some $i$ the values of $m_j$ for $j\not=i$ may not be precisely what our table specifies. For example consider $222\rightarrow4$'s Backtracking Tree letting $\mu_(S_i)=\{1,...,1,2,2,2,n-3\}$. Both ``$1...14(n-3)$" and ``$1...1222(n-3)$" appear when ``$n-3$" appears nowhere in their rows of the table. Why? Because when new elements appear $m_j$ is incremented (or decremented if Backtracking) accordingly. Thus $m_{i+1}$ is one \textit{more} than the table dictates because it appears \textit{after} the new appearance of $n-2$. Similarly, $m_{i-2}$ in ``$1...1222(n-3)$" is one \textit{less} than expected because it appears \textit{before} the new appearance of $2$.
    
    Lastly, at each turn of Backtracking we steer inside the boundaries of Maturity and terminate when we are forced to say some element appears \textit{before} its new appearance -- a contradiction by definition. The Inventory Sequence may well continue backwards further by another route but any such route can be taken only before reaching Maturity.
    
    The Backtracking Tree for $222\rightarrow 4$ tells us the edge appears (if at all) within $4$ iterations after Maturing and, at most, once in total. 
    
    The Backtracking tree for $24\rightarrow 22$ is much larger but the process of creation is identical using the table and decrementing specified above. Whereas the $222\rightarrow 4$ tree contained $9$ nodes, the $24\rightarrow 22$ tree contains $682$ nodes (though our computer-generated tree has some redundancy so the minimal nodes required for the computation might be a hundred or so smaller). The full tree and the code to produce it is available by links in the references. The tree has a height of $14$ and every path down contains at most two occurrences of $24\rightarrow 22$. We visualize here a single path from our output file:
    \begin{center}
        \includegraphics[scale=0.28]{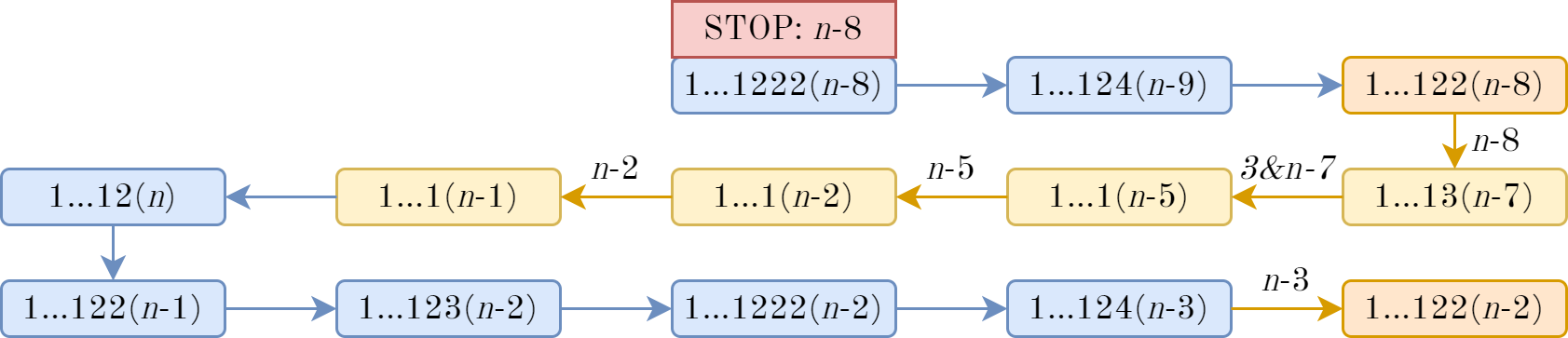}
        
        \textbf{Figure 8.3}
    \end{center}
    
    Thus with the new appearance of $2$ or $4$, the edge $24\rightarrow 22$ can occur at most once (and with the new appearance of $m_i$, at most twice). And there are a total of $1+1+2=4$ occurrences at most.
\end{proof}

We are nearly ready for the main theorem. Two lemmas are still needed. The first will tell us when and how exactly a repetition of Core Adjectives gives rise to an Inventory Loop.

\begin{lemma}
    Suppose $i\ge 2$ and for some $k\ge 1$
    $$\emptyset\not=R_i{\rightarrow}R_{i+1}{\rightarrow}...{\rightarrow}R_{i+k}=R_i{\rightarrow}R_{i+k+1}=R_{i+1}$$
    where ``$\rightarrow$" represents mapping by $g_\text{naive}.$ Then $S_{i+2}$ is a member of an Inventory Loop.
\end{lemma}
\begin{proof}
    The proof goes similarly to the second half of Theorem 7.4. In fact, we could have proven this first and used it as justification therein. However, it is often easier for readers to digest theorems built up by increasing generalization (rather than first proving the most general and going on to apply in particular cases). And the understanding of the reader more important than logical brevity.
    
    Note also that because $R_{i+k}=R_i$ \textit{and} $R_{i+k+1}=R_{i+1}$, we are assuming something stronger than the single reappearance of some $R_j$. 
    
    We may assume $1\in S_i$ since if not all elements in $\mu(S_{i-1})$ are $\ge2$. But if so, $\mu(S_{i-1})$ must contain \textit{only} $2$'s since otherwise $S_{i-1}$ would be larger in order than $S_i$ -- a contraction since by Lemma 4.2 we would have $i=1$ (but we assumed $i\ge 2$). And if $\mu(S_{i-1})=*\{2\}$, we would necessarily also have $R_i=\emptyset$ -- contradicting the lemma assumption.
    
    Given then $1\in S_i$, Lemma 7.2 ensures
    $$[S_{i}]=[S_{i+1}]=...=[S_{i+k+1}].$$
    Noting in general $|[S_j]|=|\mu(S_j)|=\text{mult}_{\mu(S_j)}(1)+|R_j|+1$ the former implies
    $$\text{mult}_{\mu(S_i)}(1)=\text{mult}_{\mu(S_{i+k})}(1)\quad\text{and}\quad \text{mult}_{\mu(S_{i+1})}(1)=\text{mult}_{\mu(S_{i+k+1})}(1).$$ 
    The first of these implies also $\text{mult}_{S_{i+1}}(1)=\text{mult}_{S_{i+k+1}}(1)$ since in general if $1\in S_j$ then $\text{mult}_{S_{j+1}}(1)=\text{mult}_{\mu(S_j)}(1)+1$. And again, Lemma 7.2 tells us $m_{i+1}=m_{i+k+1}$. Now since A) $R_{i+1}=R_{i+k+1}$, B) $\text{mult}_{\mu(S_{i+1})}(1)=\text{mult}_{\mu(S_{i+k+1})}(1)$, and C) $m_{i+1}=m_{i+k+1}$, we may conclude $\mu(S_{i+1})=\mu(S_{i+k+1})$ which, paired with $[S_{i+1}]=[S_{i+k+1}]$, implies $S_{i+2}=S_{i+k+2}$.
\end{proof}

In Figure 8.1 and Lemma 8.1 we assumed our Sequence had Matured. This includes the assumption $|S_i|\ge 16$. But some Loops occur with less than $16$ elements and our bounds at present don't apply to them. This final lemma amends their case.

\begin{lemma}
    Suppose $|S_i|<16$ and $i\ge 2$. Then for any $k>i$ such that $S_k$ is not member of any Inventory Loop and $|S_k|<16$ we may presume $k<i+2l+22$ where $l\le |[S_k]-[S_i]|$ counts appearances of new elements.
\end{lemma}
\begin{proof}
    Let $n=|\mu(S_{i-1})|$. Since $|S_i|=2n$ we know $n<8$. Our goal then is relating $\mu(S_{i-1})$ to Figure 5.1 which gives the action of $g_\text{naive}$ on $T_n^*$ -- the set of multisets of order $n$ whose elements sum to $2n$. 
    
    There are two possibilities. If no new elements appeared in $S_{i-1}$ then $\mu(S_{i-1}\in T_n^*$ since in this case $[S_{i-1}]=[S_{i-2}]$ and we may therefore say
    $$\sum_{m\in \mu(S_{i-1})}m=|S_{i-1}|=|[S_{i-2}]+\mu(S_{i-2})|=2|[S_{i-2}]|=2|[S_{i-1}]|=2n.$$
    But if any new elements \textit{did} appear then $\mu(S_{i-1})$ will (by definition) still have $n$ elements but summing instead to strictly \textit{less than} $2n$.
    
    In total, this means the sequence $\{\mu(S_j)\}_{j=i-1}^\infty$ swims around $T_n^*$ until either some set of Adjectives reappears (implying an Inventory Loop has been entered) or some new element appears and the sequence eventually jumps up into $T_{n+1}^*$ (or into $T_{n+2}^*, T_{n+3}^*, $...). Lemma 5.2 assures us as well the sequence spends at most one iteration outside its source and destination $T_m^*$ at every subsequent arrival of new elements. Thus we can bound the maximum iterations occurring before either $|\mu(S_j)|$ is at least $8$ (and therefore $|S_{j+1}|>16$) or a Loop begins. To do this let's redraw Figure 5.1 marking out the longest path available in each $T_n^*$:
    \begin{center}
        \includegraphics[scale=0.25]{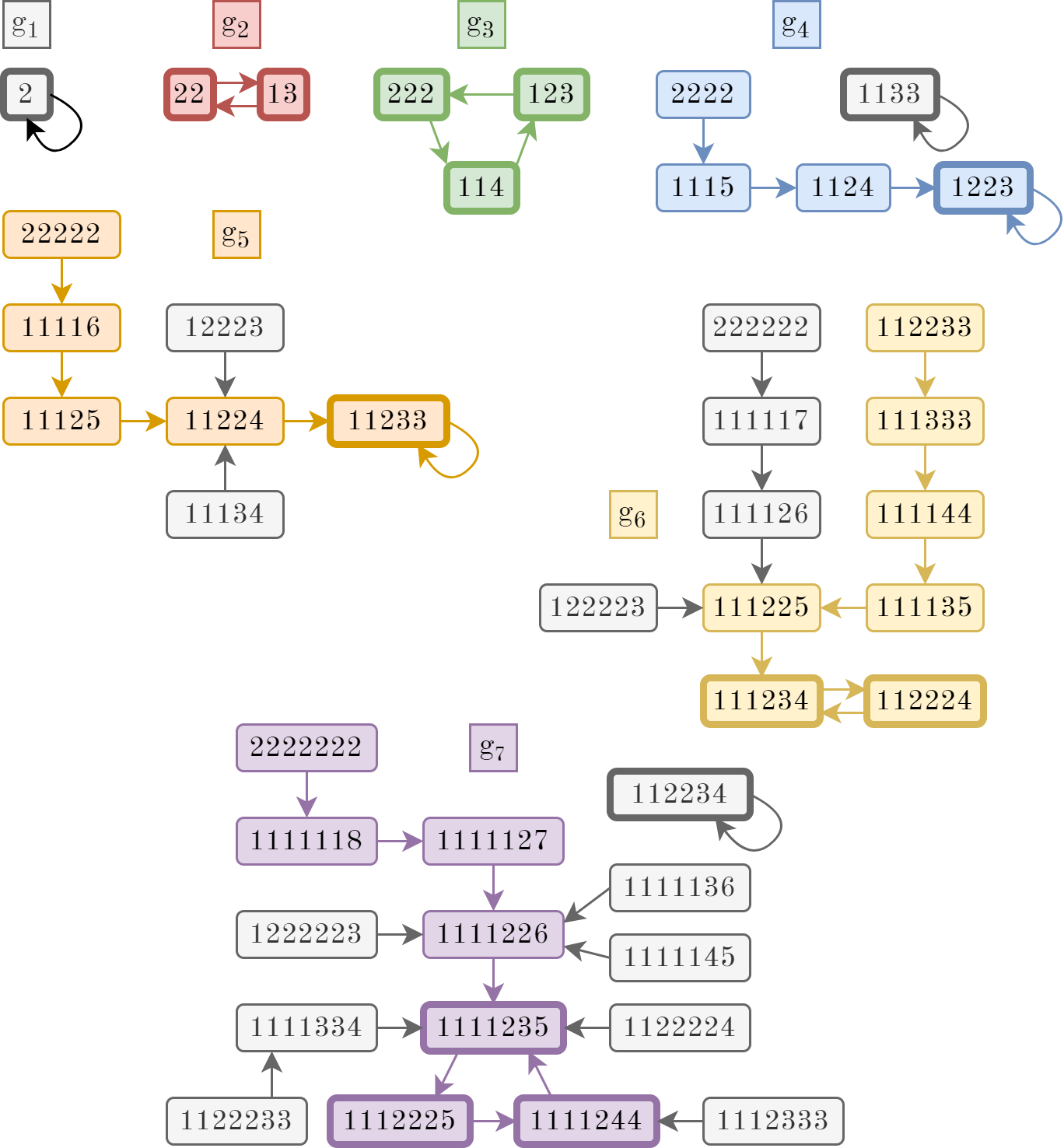}
        
        \textbf{Figure 8.4}
        
        The longest paths are left in color. Remaining nodes are colorless.
    \end{center}
    
    And we see the path lengths for $n=1,..,7$ are length $1,2,3,4,5,7,7$ respectively. Incrementing each (for the border crossing between $T_n^*$'s) and taking sum we get $2+3+4+6+6+8+8=36$ iterations at most. But instead of leaving a grand total of $36$ we instead choose a slightly different representation, attributing $2$ iterations to each new element (hence ``$2l$" in the lemma statement) and leaving $36-14=22$ as constant. This modified representation plays more nicely with the main theorem.
\end{proof}

\begin{theorem}
    The pre-period of any Inventory Sequence with initial value $S_0\in\mathbb{N}_+^*$ is at most
    $$2(\max S_1 - |[S_1]|)+\log_{\sqrt{2}}\log_{5/4}\frac{\max(c_0,10)}{8}+62$$
    where $c_0=|[\mu(S_0)]|$.
\end{theorem}
\begin{proof}
    Here we total up the iterations from the relevant theorems and lemmas:
    \begin{enumerate}
        \item Theorem 7.3 says after 
        $$k=13+\log_{\sqrt{2}}\log_{5/4}\frac{\max(c_0,10)}{8}$$
        iterations from the initial value either $S_k$ is an Inventory Loop or the Core Adjectives, $R_k$, are in Integer Notation one of $2,3,4,22,23,24,222$ or is the emptyset $\emptyset$.
        \item Lemma 8.3 says after $k=2+22+2l=24+2l$ iterations from the initial value either $S_k$ is part of an Inventory Loop or $|S_k|\ge 16$ (where $l$ counts appearances of new elements -- we will bound it further down).
        \item The former two taken together imply after $$k=24+2l+\log_{\sqrt{2}}\log_{5/4}\frac{\max(c_0,10)}{8}$$
        iterations from the initial value either $S_k$ is part of an Inventory Loop or the sequence has Matured.
        \item Figure 8.1 and Lemma 8.2 taken together imply an Inventory Sequence will enter a Loop at most $7$ iterations after Maturity if no new elements appear.
        \item One can conclude from Figure 8.1 that each new element delays the Loop by at most two Generations unless it appears with $R_k=\{2,2,2\}$ or $R_k=\{2,4\}$.
        \item Lemma 8.1 tells us the former exceptional appearances occur at most a fixed amount. In fact, we can tabulate exactly how many iterations each might add:
        \begin{center}
            \begin{tabular}{c|c|c|c|c|c}
                 & No. new & Expected & Worst & Max. no. & Worst\\
                Edge & elements & delay & delay & occurrences & overture\\ \hline
                $24\rightarrow \emptyset$ & $3$ & $6$ & $7$ & $1$ & $1$ \\
                $24\rightarrow 2$ & $2$ & $4$ & $6$ & $2$ & $4$ \\
                $24\rightarrow 22$ & $1$ & $2$ & $5$ & $4$ & $12$ \\
                $222\rightarrow 3$ & $2$ & $4$ & $6$ & $1$ & $2$ \\
                $222\rightarrow 4$ & $1$ & $2$ & $6$ & $1$ & $4$ \\
                $222\rightarrow 23$ & $1$ & $2$ & $4$ & $1$ & $2$ \\
            \end{tabular}
        \end{center}
        Thus attributing a delay of $2$ iterations per new element gives bound off by no more than $1+4+12+2+4+2=25$ iterations. More careful analysis could reduce this number but, as there are a lot possible improvements to our bound's constant term, we will enumerate them all together in the next section.
        \item Adding the constant terms of items (3.), (4.), and (6.) gives $24+7+25=56$. Thus after 
        $$k=2l+\log_{\sqrt{2}}\log_{5/4}\frac{\max(c_0,10)}{8}+56$$
        iterations from the initial value, our Sequence has entered a Loop.
    \end{enumerate}
    
    Our last task is to get rid of the ``$l$". This is possible because the elements all Inventory Sequences are bounded from above. In particular, Theorem 5.4 tells us the largest element appears by $S_3$ (except for a handful of exceptions). In other words, $\max S_i=\max S_3$ for $i\ge 3$. Further, Corollary 4.1.1 states the Height of an Inventory Sequence at most increments. Thus $S\max S_3\le \max S_1+2$ and in most cases $\max S_i\le \max S_1 + 2$ for all $i\ge 0$. The worst of the exceptions is $S_0=\{1\}$ which gives $\max S_1=1$ and $\max S_8=\max S_i=4$ for $i\ge 8$. So in all cases $\max S_i\le \max S_1 + 3$. 
    
    Now $l$ counts new appearances from $S_2$ onward. So because the elements of the Sequence are bounded by $\max S_1+3$ we may conclude
    $$l\le \max S_1 + 3 - |[S_1]|.$$
    Substituting into our former bound we obtain a pre-period bound of
    $$2(\max S_1 +3- |[S_1]|)+\log_{\sqrt{2}}\log_{5/4}\frac{\max(c_0,10)}{8}+56$$
    $$=2(\max S_1 - |[S_1]|)+\log_{\sqrt{2}}\log_{5/4}\frac{\max(c_0,10)}{8}+62.$$
\end{proof}

\begin{corollary}
    The pre-period of any Inventory Sequence with $S_0\in\mathbb{N}_+^*$ is at most $2\max S_1+60.$
\end{corollary}
\begin{proof}
    The recurrence $S_1=\mu(S_0)+[S_0]$ tells us $c_0=|[\mu(S_0)]|\le |[S_1]|$. This means the larger we make $c_0$, the larger $|[S_1]|$ becomes too. In particular, an increment to $c_0$ increases the $\log\log$ term by at most
    $$\log_{\sqrt{2}}\log_{5/4}\frac{11}{8}=1.026...$$
    but reduces the $2(\max S_1 - |[S_1]|)$ term by $2$. In other words, the bound is largest when $|[S_1]|$ is smallest.
    
    Picking then $|[S_1]|=1$ gives us a more symbolically compact bound:
    $$2(\max S_1 - 1)+\log_{\sqrt{2}}\log_{5/4}\frac{\max(1,10)}{8}+62=2\max S_1 -2 + 0 + 62 = 2\max S_1 + 60.$$
\end{proof}

\section{On Improvements}
It isn't clear the pre-period bound is best formed as $a\max S_1+b$. And ours ($a=2$ and $b=60$) isn't necessarily the tightest. However, Bronstein and Fraenkel's original request was for \textit{meaningful} -- not \textit{exact} -- pre-period bounds and we feel the request is answered. Before closing, the authors would like to speculate about improvements to our bound.

Firstly, $b=60$ is easily improvable and the authors took no pains to do so when the mathematical ecosystem offered brevity in exchange. The following shortcuts each admit improvements:
\begin{enumerate}
    \item Lemma 6.2 might be adapted to our particular subspecies of multisets yielding better bounds -- or more likely, a set of bounds for different cases (e.g. when new elements have appeared and when not). The result would carry into Lemmas 6.3 and 6.4 and become a lower constant term and/or lower logarithmic bases in Theorem 5.4. A lowered constant would also reduce the multisets listed in Corollary 7.1.1 and would then become a lower constant in Theorem 7.3.
    \item Lemma 8.1 was completed as if all $13$ edges could occur together at their maximum occurrence. However, most of them are exclusive. For example,  $2\rightarrow\emptyset,\ 23\rightarrow\emptyset,$ and $24\rightarrow\emptyset$ each require the new appearance of $2$ and thus only one can appear. Similarly if $24\rightarrow 22$, this could probably be shown to occur at most $2$ or $3$ times (instead of $4$). Thus the constant ``$25$" from part (6.) of Theorem 8.3's proof can probably be shrunk to $10$ or so.
    \item Lemma 8.2 requires the repetition of a consecutive Core Adjective pair. However, if adapted for our case, it is possible $222\rightarrow 24\rightarrow 222$ or $24\rightarrow 222\rightarrow 24$ (i.e. the repetition of a single Core Adjective multiset) would be enough to force a Loop. This would reduce the delay constant from $7$ to $6$ as well as reduce the overture of some edges in part (6.) of Theorem 8.4's proof. 
    \item Lemma 8.3 was completed as if the longest path of all seven $T_n^*$'s might occur together. However, the length of the path through $T_{n-1}^*$ has consequences on the path taken through $T_n^*$. For example, if the longest path (length $5$) is taken though $T_5^*$ then the path through $T_6^*$ will be at most length $3$. The constant ``$22$" can probably be reduced to $14$.
\end{enumerate}

But these improvements still leave us with something $2\max S_1+b$ in the end (around $b=30$ probably). One might then wonder can $a=2$ be decreased?

No. Not in the general case $S_0\in\mathbb{N}_+^*$ at least. The orange edges in Figure 8.1 can in fact be taken indefinitely. For example, $S_0=4\{4,k,k+1\}$ takes $22\rightarrow 3$ repeatedly ($k-6$ times exactly):
    \begin{center}
        \includegraphics[scale=0.28]{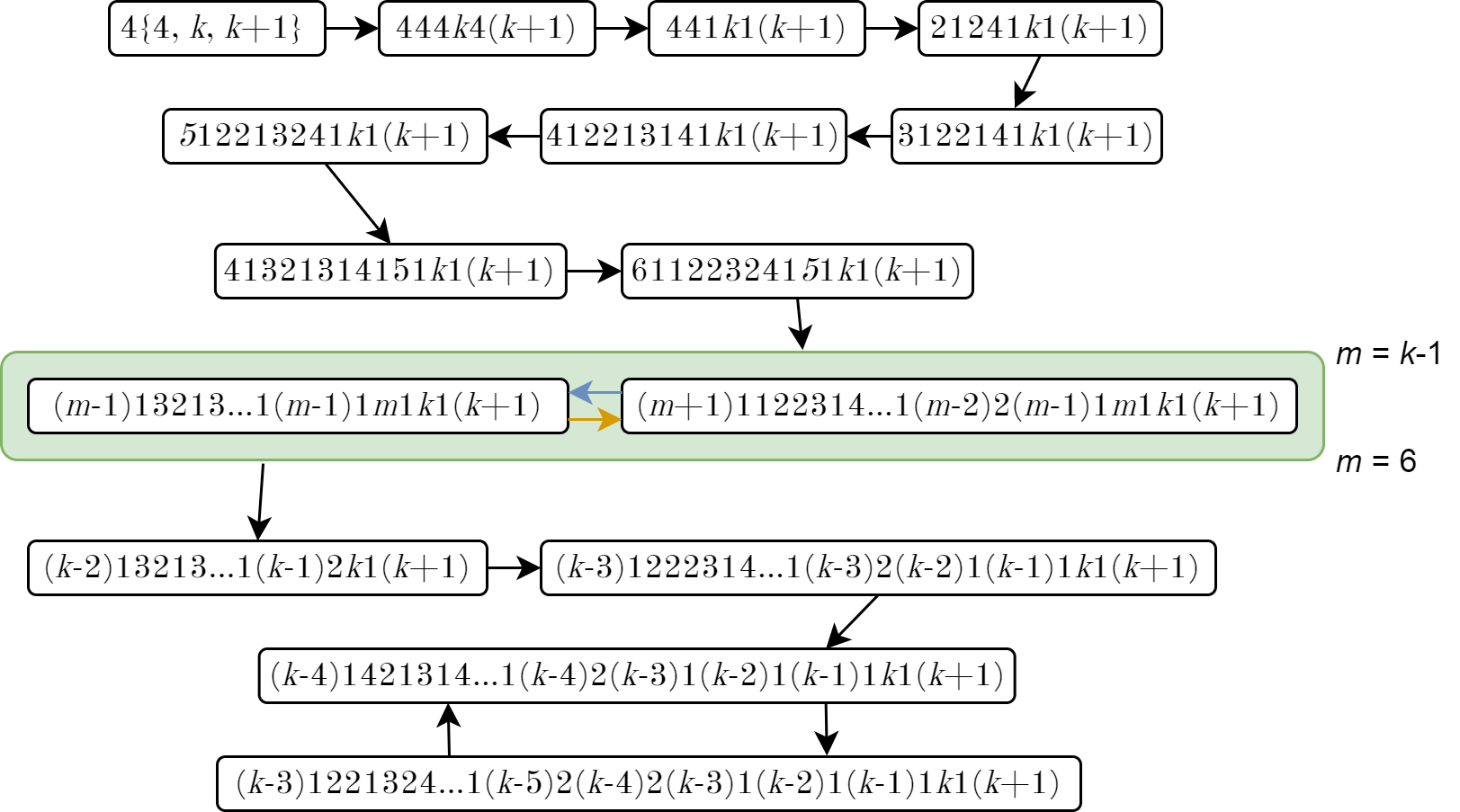}
        
        \textbf{Figure 9.1}
        
        The starting value is represented as a multiset but for readability remaining nodes are written in Integer Notation.
    \end{center}

The trick works because the family bounces around the green area repeatedly -- $O(k)$ times in fact. The authors worked out similar staring values for three other orange edges. Two of them ($22\rightarrow 3$ and $23\rightarrow 22$) give rise to infinite families of starting values for which the $a=2$ is sharp. Here are the edges marked out in the Core Adjectives graph:
\begin{center}
    \includegraphics[scale=0.35]{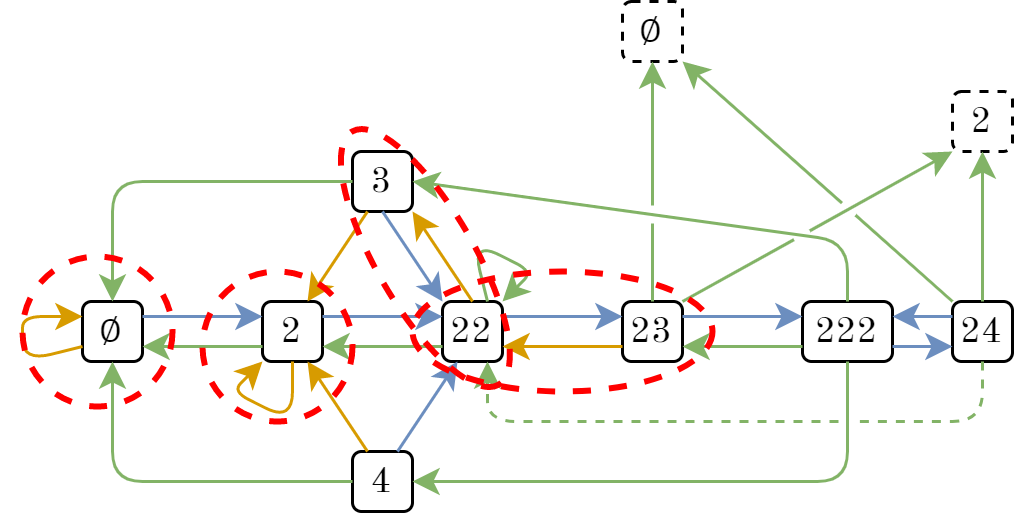}
    
    \textbf{Figure 9.2}
\end{center}
And here are their families given explicitly:
\begin{center}
    \begin{tabular}{c|c|c|c|c}
        $S_0$ & $k\ge$ & Edge & $\max S_1$ & Loop at \\ \hline
        $4\{4,k,k+1\}$ & $7$ & $22\rightarrow3$ & $k+1$ & $2k-2=2\max S_1-4$ \\
        $3\{3\}+2\{k,k+1,k+2\}$ & $5$ & $23\rightarrow22$ & $k+2$ & $2k-2=2\max S_1 -6$ \\
        $\{2,2,k,k,k+1,k+2\}$ & $5$ & $2\rightarrow2$ & $k+2$ & $k+2=\max S_1$ \\
        $k\{k+1\}$ & $7$ & $\emptyset\rightarrow\emptyset$ & $k$ & $k+4=\max S_1+4$ \\
    \end{tabular}
\end{center}
    
It is left to the reader to work out families for $3\rightarrow 2$ and $4\rightarrow 2$ (if they exist!).

\section{Replacing $\mathbb{N}_+$ with $\mathbb{N},\mathbb{Z}$}

So far we have assumed our multisets are of strictly positive integers. We'd like to know if our results hold in other regions -- say $\mathbb{N}$ or $\mathbb{Z}$. It turns out by simply allowing $0$ and $-1$ (or really any two values outside the image of $\mu$) we lose even "Ultimate Cyclicity". That is to say, one can find $S_0\in\mathbb{N}\cup \{-1\}^*$ which never enters a Loop. For example:
\begin{center}
    \includegraphics[scale=0.25]{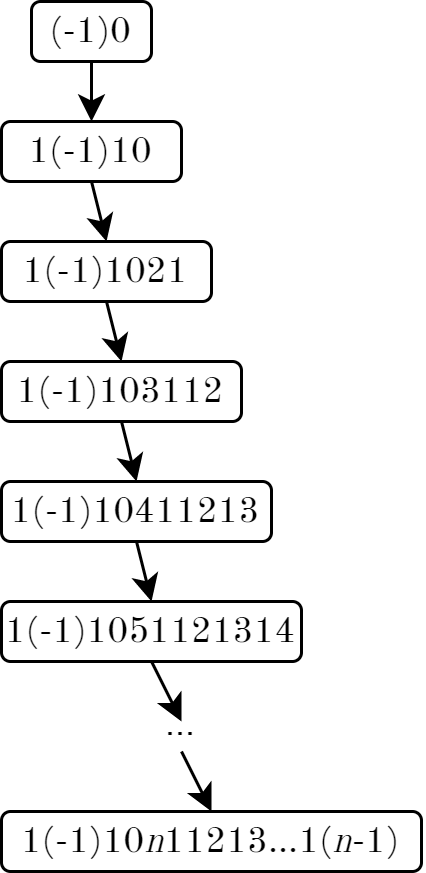}
    
    \textbf{Figure 10.1}
\end{center}

Thus $\mathbb{N}=\mathbb{N}_+\cup\{0\}$ is something of a boundary case. As far as the authors can see, the results of Sections 5 through 9 run isomorphically substituting ``$\mathbb{N}$" for ``$\mathbb{N}_+$". Section 4 is probably affected only in that the list of exceptions for Theorem 4.5 becomes longer. This is because the proof of Lemma 4.1 relied on the fact that $\max S\ge |[S]|$ for $S\in\mathbb{N}_+^*$. But in $\mathbb{N}$ we might have say $R=\{0,1,2,3,4\}$. The difficulty can be amended with some extra work.

But the authors have not pursued this variation or any others as the project has already grown larger than the authors should have allowed. Further analysis is therefore left undone and the paper will close with speculation about an interesting (but yet unstudied) reformulation and the variations it offers.

\section{Multisets as Functions}
Every multiset \enquote{$S$} corresponds to a multiplicity function \enquote{$\text{mult}_S$}. Some fun variations are easier to find (and rigorously define) by treating the Inventory Game as an iteration of functions,
$$\text{mult}_{S_0}\rightarrow\text{mult}_{S_0}\rightarrow \text{mult}_{S_0}\rightarrow ...,$$
instead of an iteration of multisets 
$$S_0\rightarrow S_1\rightarrow S_2\rightarrow ...$$
It will be a bit of work to redefine things, but the authors found the results worth the trouble.

How then do we translate our recurrence \enquote{$S_{i+1}=\mu(S_i)+[S_i]$} into the language of functions? The task can be broken down to redefining 1) the multiset-of-multiplicities function \enquote{$\mu$}, 2) the multiset-to-set function \enquote{$[\cdot]$}, and 3) multiset addition \enquote{$+$}.

Taking the first point first, what is $\text{mult}_{\mu(S)}(x)$ for a particular $x$? Well -- $x$ appears in $\mu(S)$ once for however many elements have a multiplicity of $x$ in $S$. In other words, $\text{mult}_{\mu(S)}(x)$ counts how many elements the function $\text{mult}_{S}$ sends to $x$. This value will come up a lot so we should make some notation for it (it's cluttered to express it as summation).

For any function $\sigma:X\rightarrow Y$ let \enquote{$\sigma^{-1}(x)$} mean \enquote{the set of elements sent to $y$ by $\sigma$}. In other words,
$$\sigma^{-1}(y)=\{x\in X:\sigma(x)=y\}.$$
As an equation we then have
$$\text{mult}_{\mu(S)}(x)=|\text{mult}_S^{-1}(x)|.$$

But this definition isn't quite right still. There is an exception. Zero. There are infinitely many elements in any finite multiset $S$ occurring zero times. That is, there are infinitely many elements $\text{mult}_{\mu(S)}$ sends to zero. But we exclude this. We do not, for instance, say  $\{1,3,1,8\}$ has zero $0$'s, two $1$'s, zero $2$'s, one $3$, zero $4$'s, zero $5$'s, and so on... The zeros are left out. Our equation therefore requires amendment:
$$\text{mult}_{\mu(S)}(x)=\begin{cases}|\text{mult}_S^{-1}(x)| & \text{if }x>0 \\ 0 & \text{if }x=0\end{cases}.$$

The function $\text{mult}_{[S]}$ is easier to work through. The value of $\text{mult}_{[S]}(x)$ is $1$ if $x$ appears in $S$ and is $0$ otherwise. It's therefore something of an indicator function. As an equation,
$$\text{mult}_{[S]}(x)=\begin{cases}1 & \text{if }\text{mult}_{S}(x)>0 \\ 0 & \text{if }\text{mult}_{S}(x)=0\end{cases}.$$

Lastly, addition of functions is defined simply by adding values element-wise since for any two multisets $R,S$
$$\text{mult}_{R+S}(x)=\text{mult}_R(x)+\text{mult}_S(x).$$

All together we have
$$\text{mult}_{S_{i+1}}(x)=\begin{cases}|\text{mult}_{S_i}^{-1}(x)| & \text{if }x>0 \\ 0 & \text{if }x=0\end{cases} \quad+\quad \begin{cases}1 & \text{if }\text{mult}_{S_i}(x)>0 \\ 0 & \text{if }\text{mult}_{S_i}(x)=0\end{cases}.$$
And $\text{mult}_{S_{i+1}}$ has been defined totally in terms of $\text{mult}_{S_{i}}$. To make reading easier from here on (and to clear multiset conceptions out of our minds) the function \enquote{$\text{mult}_{S_i}$} will simply be called \enquote{$\sigma_i$}. Thus the Inventory Game really means the sequence $(\sigma_0, \sigma_1,\sigma_2,...)$ for some $\sigma_0:\mathbb{N}\rightarrow\mathbb{N}$ where 
$$\sigma_{i+1}(x)=\begin{cases}|\sigma_i^{-1}(x)| & \text{if }x>0 \\ 0 & \text{if }x=0\end{cases} \quad+\quad \begin{cases}1 & \text{if }\sigma_i(x)>0 \\ 0 & \text{if }\sigma_i(x)=0\end{cases}.$$

Our previous results then tell us that if $\sigma_0$ sends only a finite portion of $\mathbb{N}$ to non-zero values then the sequence $\{\sigma_i\}_{i=0}^\infty$ enters a cycle of length $1,2,$ or $3$ in $O(M)$ time where $M=\max S_1=\max \mathbb{N}/\sigma_1^{-1}(0)$. 

Before going on to exotic variations, it may be good to work an example. Let's take the classic case $S_0=\{1\}$ with the corresponding function
$$\sigma_0(x)=\text{mult}_{S_0}(x)=\begin{cases}1&\text{if }x=1\\0&\text{oth.}\end{cases}.$$

Because case equations are large, difficult to read (and to code), and are usually ugly, we will use an alternative representation. The statement \enquote{$\sigma_0:1\rightarrow1,*\rightarrow 0$} will mean \enquote{$\sigma_0$, the function sending $1$ to $1$ and everything else to $0$}. Thus our iteration may be visualized in $3$ ways -- as multisets, as function descriptions, and as function compositions:
\begin{center}
    \includegraphics[scale=0.25]{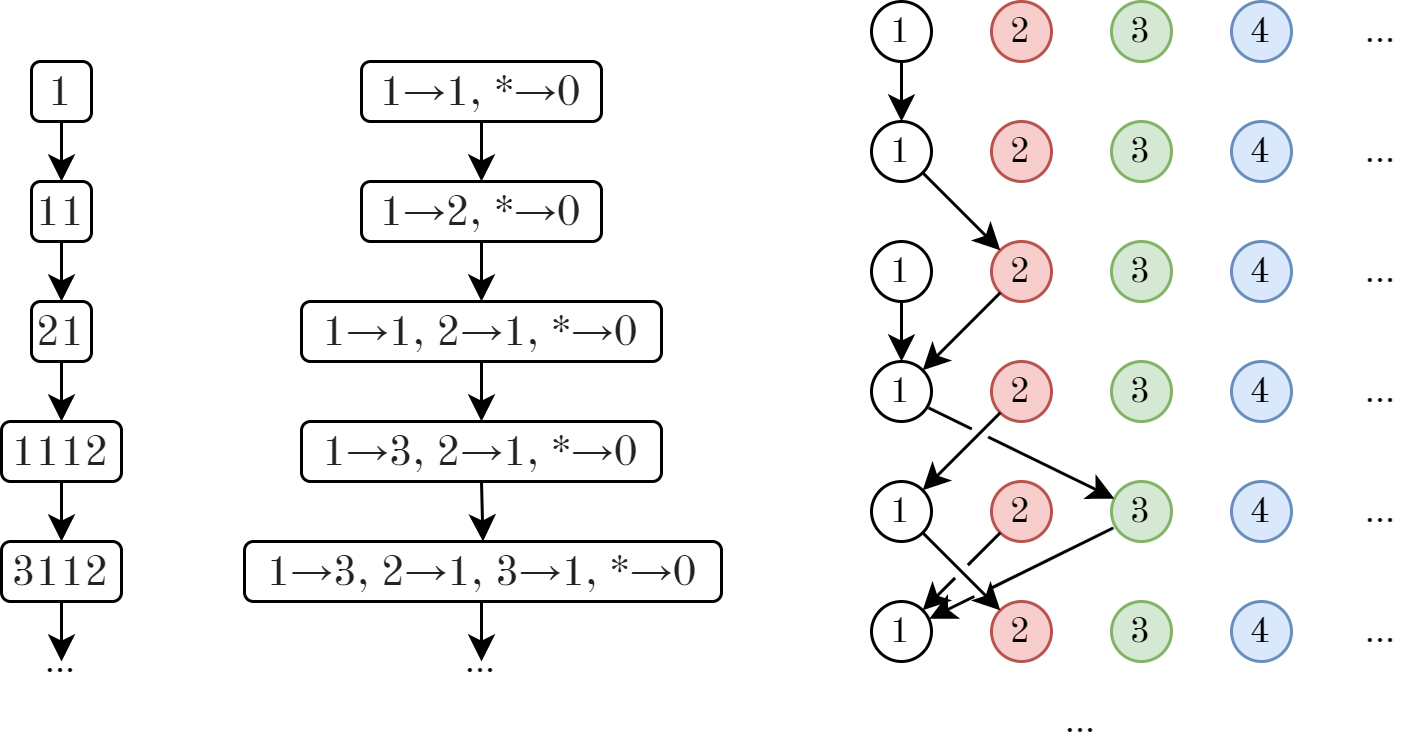}
    
    \textbf{Figure 11.1}
\end{center}
Of course, these redefinitions and revisualizations are only like setting the table. Now we are ready to eat.

It was said in the first section some mathematicians have played the Inventory Game with infinitely long starting values [1,2,3]. In our current language, we would say they chose $\sigma_0$ sending infinitely many values to a non-zero image. But to keep everything well-defined, they had to make sure infinitely many values were not sent to any particular $x>0$ since then $|\sigma^{-1}(x)|$ would be undefined. But there \textit{are} perfectly reasonable ways to define $|X|$ when the set $X$ is infinitely large. And doing so gives us transfinite variations of the Inventory Game.

In the most simple case, we take $\hat{\mathbb{N}}=\mathbb{N}\cup \{\infty\}$ instead of $\mathbb{N}$ and say $|\sigma^{-1}(x)|=\infty$ whenever $\sigma$ sends infinitely many elements to $x$ (to be completely rigorous we should also specify that $\infty>0$ and $\infty+1=\infty$). For example, if we start with $S_0=\mathbb{N}_+$ (or equivalently with $\sigma_0:0\rightarrow 0,*\rightarrow 1$) then we get the sequence on the left:
\begin{center}
    \includegraphics[scale=0.25]{"TransfiniteVariations".png}
    
    \textbf{Figure 11.2}
\end{center}
And a Loop occurs at $(\sigma_8=\sigma_{10}: 1\rightarrow \infty, 2\rightarrow 2, 4\rightarrow 2, \infty\rightarrow 2, *\rightarrow 1)$.

There are other more interesting Transfinite variations (e.g. the one on the right). So to keep from confusing ourselves, let's call the version on the left STIG for Simplest Transfinite Inventory Game. The version we're about to define (the one on the right) will be called TOIG for Transifnite Ordinals Inventory Game since it is played with surreal numbers (hence using \enquote{$\omega$} instead of \enquote{$\infty$}).

In STIG, we said \enquote{$\infty=\infty+1$}. TOIG is more or less the result of saying  instead \enquote{$\infty\not=\infty+1$}. Correspondingly in STIG, we said $|\sigma^{-1}(x)|=\infty$ if $\sigma$ sent any infinite portion of $\mathbb{N}$ to $x$. In TOIG however, we say $|\sigma^{-1}(x)=\omega$ if $\sigma$ sends all of $\mathbb{N}$ to $x$. And if $\sigma$ sends all but a finite portion of $\mathbb{N}$ to $x$, say $\mathbb{N}/P$, then we say $|\sigma^{-1}(x)|=\omega-|P|$. In the round of TOIG in Figure 11.2 for example, $\sigma_1$ sends $\mathbb{N}/\{0,1\}$ to $1$. Accordingly, $|\sigma_1^{-1}(1)|=\omega-2$ and $\sigma_2(1)=|\sigma_1^{-1}(1)|+1=\omega-1$. And again, if in addition to $\mathbb{N}/P$, $\sigma$ sends some finite set of transfinites, $Q$, to $x$ then we say
$$|\sigma^{-1}(x)|=\omega-|P|+|Q|.$$

But this variation has a couple of holes the authors have not seriously attempted to patch. Firstly, what happens when infinite transfinites are sent to $x$? And secondly, what happens when an infinite portion of $\mathbb{N}$ with infinite exceptions is sent to $x$? For instance of the latter, what if we start with $S_0=\{0,2,4,6,8,...\}$? (i.e. $\sigma_0^{-1}(1)=2\mathbb{N}$ -- the non-negative even numbers). 

The only hopeful solution occurring to the authors was to play over Conway's surreal numbers (as opposed to a more simple surordinal set). The first difficulty might then be resolved setting $\sigma^{-1}(x)$ to $2\omega$ or $\omega^2$ (or perhaps something more nuanced). The second difficulty, $\sigma_0^{-1}(1)=2\mathbb{N}$, might be also resolved setting $|\sigma_0^{-1}(1)|$ to $\frac{1}{2}\omega$. And in general if $X$ has density $d$ in $\mathbb{N}$ then we set $|X|=d\omega$. If the density is $0$ or $1$ then an interesting work-around might be found with $\epsilon=\frac{1}{\omega}$ (or the variation might play out consistently without any patch-work). The soil seems fertile here but the authors haven't made time to plant anything.

The multiset-as-function formulation has also some nice alterations worth mentioning. But it will need some surgery first -- it's not flexible enough currently. There are three procedures to be done:
\begin{enumerate}
    \item First we treated $\sigma_i$ as a map from $\mathbb{N}$ to $\mathbb{N}$, then as a map from $\hat{\mathbb{N}}$ to $\hat{\mathbb{N}}$, and thirdly as a map from the surreal numbers $\mathbb{S}$ to themselves. In general, $\sigma_i$ sends some set $G$ \enquote{into} itself. Really the only requirement of $G$ is an additive structure with identity -- in other words, to clearly define $\sigma_i:G\rightarrow G$, we need some $+:G^2\rightarrow G$ and $0\in G$ such that $0+x=x$ for any $x\in G$.
    \item When working out $\text{mult}_{\mu(S)}$ we had to make an exception for zero since we don't mention zero amounts of things in description (e.g. saying $13$ has \enquote{one $1$, zero $2$'s, one $3$, zero $4$'s, zero $5$'s, ...}). In spoken languages, zero is (usually) an assumed or insignificant Adjective. But what if we want to mention zero things explicitly? -- if we want to give zero some significance? Or alternatively, what if we want to treat other Adjectives besides zero insignificantly? Both of these can be done defining a set $I\subset G$ of \textit{Insignificant Adjectives} where so far we have used $I=\{0\}$.
    \item We said in the first section some people play the Inventory Game without Nouns. Their gameplay can be accommodated by adding a parameter $r\in G$ counting Noun-mentions. We played with $r=1$. Choosing $r=0$ gives the Nounless Variation and corresponds tot he map $S\rightarrow \mu(S)$ in the multiset formulation. More will be said about $r>1$ further on.
\end{enumerate}

Putting all these together, the recurrence becomes
$$\sigma_{i+1}(x) = \begin{cases}|\sigma_i^{-1}(x)|& \text{if } x\not\in I \\ 0 & \text{if }x\in I\end{cases} \quad+\quad \begin{cases}r & \text{if } \sigma_i(x)\not\in I\\ 0 & \text{if } \sigma_i(x)\in I\end{cases}.$$
The order operator $|\cdot|$ can really be any map sending subsets of $G$ to elements of $G$ (i.e. $|\cdot|:\mathcal{P}(G)\rightarrow G$).

A lot of variations can now be labeled by a $4$-tuple $(G, I, |\cdot|, r)$. Our main specimen of analysis has been $(\mathbb{N}, \{0\}, |\cdot|, 1)$. STIG is $(\mathbb{N}\cup \{\infty\}, \{0\}, |\cdot|, 1)$ where $|\cdot|$ sends infinite sets to $\infty$ and TOIG is $(\mathbb{S}, \{0\}, |\cdot|, 1)$ where $\mathbb{S}$ is Conway's surreal numbers and $|\cdot|$ we never defined exhaustively. Additionally, the decimal Nounless Variation is $(\mathbb{Z}/10\mathbb{Z}, \emptyset, |\cdot|, 0)$ where $|\cdot|$ sends sets of order $10$ to $0$. This variation has two Loops. There is a $1$-cycle corresponding to the self-descriptive number $6210001000$ and a $2$-cycle corresponding the mutually descriptive pair $7101001000$ and $6300000100$. The authors here resist the temptation to speculate about Nounless Transfinite Variations.

One wonders how variations behave when $I$ is neither $\emptyset$ nor $\{0\}$. Consider then $(\mathbb{N}, \mathbb{N}/\{1,2,3\}, |\cdot|, 1)$. It's a lot like the version we analyzed except a Noun is mentioned only if there are $1,2,$ or $3$ of it. For example, we would say $1133335255$ has \enquote{two $1$'s, one $2$, and three $5$'s}. The $3$'s are not mentioned. This variation produces Loops longer than $3$ elements:
$$1$$
$$11$$
$$21$$
$$1112$$
$$3112$$
$$211213$$
$$\mathbf{312213}$$
$$212223$$
$$1113$$
$$3113$$
$$2123$$
$$112213$$
$$\mathbf{312213}$$
$$...$$

The last point of any substance the authors have to make is about $r>1$. Let's play $(\mathbb{N}, \{0\}, |\cdot|, 3)$ -- i.e. the variation we analyzed with $r=3$. This means Nouns are repeated $3$ times each in description saying, for example, $1381$ has 
\begin{center}\enquote{two $1-1-1$'s, one $3-3-3$, and one $8-8-8$.}\end{center}
One inevitably thinks of speaking with a stutter. But we will call these OEIGs for Over-Emphasized Inventory Games (and spoofing the OEIS [10]). Thus here we say $f(1381)=211113331888$ and a round might go like:
$$1$$
$$1111$$
$$4111$$
$$3111\ 1444$$
$$4111\ 1333\ 3444$$
$$4111\ 4333\ 4444$$
$$3111\ 3333\ 6444$$
$$3111\ 5333\ 3444\ 1666$$
$$4111\ 5333\ 1555\ 3444\ 3666$$
$$4111\ 5333\ 4555\ 4444\ 3666$$
$$3111\ 4333\ 4555\ 6444\ 3666$$
$$3111\ 5333\ 3555\ 5444\ 4666$$
$$\mathbf{3111\ 5333\ 5555\ 4444\ 3666}$$
$$\mathbf{3111\ 5333\ 5555\ 4444\ 3666}$$
$$...$$

The gameplay should feel oddly similar to the $r=1$ variation we studied. Look -- the Loop's Adjectives are $33455$ (a bit similar to the fixed point of $T_5^*$: $11233$). It isn't coincidence. Alternatively, if one had started with $S_0=\{1,3,8,1\}$ (or again, equivalently with $\sigma_0:1\rightarrow 2,3\rightarrow 1,8\rightarrow 1,*\rightarrow 0)$ one would find a $2$-cycle:
$$ 3111\ 3222\ 7333\ 6444\ 3555\ 3666\ 3777\ 4888$$
$$ 3111\ 3222\ 8333\ 4444\ 3555\ 4666\ 4777\ 3888$$
with Adjectives again \textit{very} similar to $T_8^*$:
\begin{center}
    \begin{tabular}{c|c}
        Loop Adjectives & $2$-cycle \\
        from $r=3$ var. & from $T_8^*$ \\ \hline
        $33333467$ & $11111245$ \\
        $33334448$ & $11112226$
    \end{tabular}
\end{center}
To find out what's going on, we should generalize $T_n^*$. So let $T_{n,r}^*$ be the set of multisets
\begin{enumerate}
    \item containing $n$ elements,
    \item whose elements sum to $r+1$ times their order,
    \item and whose elements are $r$ or greater.
\end{enumerate}
Formally 
$$T_{n,r}^*=\{S\in\mathbb{N}_{\ge r}:\sum_{x\in S}x=(r+1)|S|=(r+1)n\}.$$
Our old definition then lines up as $T_n^*=T^*_{n,1}$. One could similarly generalize $g_n$ and $\mu_+$ and go on to show A) the Adjectives of a Loop in $(\mathbb{N}, \{0\}, |\cdot|, r)$ are members of $T_{n,r}^*$, B) the multisets of any $T_{n,r}^*$ and $T_{n,t}^*$ are in bijective correspondence (in particular if $t>r$ then any $S\in T_{n,t}^*$ is obtained adding $t-r$ to each element of some $R\in T_{n,r}^*$), and C) therefore every Loop of $(\mathbb{N}, \{0\}, |\cdot|, r)$ is mapped under a bijective correspondence to a cycle in $T_{n,0}^*$ under $\mu$. Note also $T_{n,0}^*$ is in bijective correspondence to partitions of $n$.

The authors wonder if in general the Loops of $(G, I, |\cdot|, r)$ can always be reduced by isomorphism to the Loops of $(G, I, |\cdot|, \hat{r})$ for some particular $\hat{r}\in G$ (in the former case, $\hat{r}=1$). They wonder also if and when such isomorphisms of Loops offer a translation of pre-period bounds across variations.

The speculations have been completed. The authors have navigated only one path and a couple landmarks of the Inventory forest (and perhaps not the most efficient path at that). They are not sure if these speculations are so simple (or so general) as to be useless. That is, they aren't sure if they have lead the reader to really new and nice scenery -- or only in a circle (or perhaps to a cliff). Thus the length of these speculations may only be evidence they don't have the nose for telling dead from fertile mathematical soil. But they believe it is always better to ask more questions than one answers and that to live in a universe with no open questions (or more likely, to live without interest for the open questions around oneself) is the closest experience to Hell we might have while alive. So, as one might imagine the totality of known mathematics as a house and the unknown as the wild outdoors, the authors have in this last section thrown open the doors and windows nearest them. They hoped to give the reader new exploration (or at least a fresh breeze). They ask the reader's forgiveness if -- in ignorance or haste -- they have only thrown open a closet, the oven, or the door to another room.

\end{document}